\documentclass[11pt]{article}
\usepackage{amsmath,amsfonts,amsthm,amssymb, mathtools}
\usepackage{mathrsfs,graphicx,color,latexsym, tikz,calc,cite,enumerate,indentfirst}
\usetikzlibrary{shadows}
\usepackage{tikz-cd}
\usetikzlibrary{patterns,arrows,decorations.pathreplacing}
\usepackage[colorlinks=true,
linkcolor=blue,citecolor=blue,
urlcolor=blue]{hyperref}

\newtheorem{theorem}{Theorem}[section]
\newtheorem{lemma}[theorem]{Lemma}

\newtheorem{corollary}[theorem]{Corollary}
\newtheorem{remark}[theorem]{Remark}

\newtheorem{example}[theorem]{Example}
\newtheorem{definition}[theorem]{Definition}

\renewcommand\proofname{\bf{Proof}}

\allowdisplaybreaks

\title{\bf Combinatorial Laplacians and Relative Homology of Complex Pairs}

\author{
Xiongfeng Zhan$^a$,  \ Xueyi Huang$^{a,}$\thanks{Corresponding author.},  \  Lu Lu$^b$ \setcounter{footnote}{-1}\footnote{\textit{Email address:} zhanxfmath@163.com (X. Zhan), huangxy@ecust.edu.cn (X. Huang), lulumath@csu.edu.cn (L. Lu).}\\[2mm]
	\small $^a$School of Mathematics, East China University of Science and Technology, \\
	\small  Shanghai 200237, P. R. China\\
	\small $^b$School of Mathematics and Statistics, Central South University,\\
	\small Changsha, Hunan, 410083, P. R. China
}

\date{}

\begin{document}
	\maketitle
	\begin{abstract}

As a discretization of the Hodge Laplacian, the combinatorial Laplacian of simplicial complexes has garnered significant attention. In this paper, we study combinatorial Laplacians for complex pairs $(X, A)$, where $A$ is a subcomplex of a simplicial complex $X$. We establish a relative version of the matrix-tree theorem for complex pairs, which generalizes both the matrix-tree theorem for simplicial complexes proved by Duval, Klivans, and Martin (2009) and the result for Dirichlet eigenvalues of graph pairs by Chung (1996). Furthermore, we derive several lower bounds for the spectral gaps of complex pairs and characterize the equality case for one sharp lower bound. As by-products, we obtain sufficient conditions for the vanishing of relative homology. Our results demonstrate that the combinatorial Laplacians for complex pairs are closely related to relative homology.

		\par\vspace{2mm}
		
		\par\vspace{1mm}
		
	\end{abstract}

	\section{Introduction}\label{section::1}
	
	A \emph{simplicial complex} $X$ is a collection of finite sets that is closed under set inclusion; that is, if $\sigma \in X$ and $\tau \subseteq \sigma$, then $\tau \in X$. A finite set $\sigma$ in $X$ is called a \emph{face} or a \emph{simplex} of $X$, and its \emph{dimension} is defined as $\dim \sigma = |\sigma| - 1$. The union of all faces of $X$ is called the \emph{vertex set} of $X$, denoted by $V(X)$.  Note that the empty set $\emptyset$ is always a face of $X$, and its dimension is $-1$. The \emph{dimension} of $X$, denoted by $\dim X$, is the maximum dimension among all its faces. The faces that are maximal under inclusion are called \emph{facets}. A simplicial complex is \emph{pure} if all its facets have the same dimension. 

A \textit{subcomplex} of $X$ is a subset of $X$ that is itself a simplicial complex. For any integer $p \leq \dim X$, the $p$-dimensional \textit{skeleton} of $X$, denoted by $X_{(p)}$, is the subcomplex consisting of all faces of dimension at most $p$. In particular, we denote by $G_X$ the $1$-skeleton of $X$, viewed as a graph. A \textit{missing face} of $X$ is a subset $\sigma \subseteq V(X)$ such that $\sigma \notin X$ but every proper subset $\tau \subsetneq \sigma$ belongs to $X$. The maximal dimension of a missing face of $X$ is denoted by $h(X)$. Let $X_k$ denote the set of all $k$-dimensional faces of $X$. For any $\sigma \in X_k$, the \textit{degree} of $\sigma$ in $X$ is defined as 
$\deg_X(\sigma) = \bigl|\{\eta \in X_{k+1} \mid \sigma \subseteq \eta\}\bigr|$,
and we denote by $\sigma_{k-1}$ the set $\{\tau \in X_{k-1} \mid \tau \subseteq \sigma\}$.

Let $X$ and $Y$ be two simplicial complexes. A \textit{simplicial map} from $X$ to $Y$ is a function $f \colon V(X) \to V(Y)$ such that $f(\sigma) \in Y$ for all $\sigma \in X$. We say that $X$ and $Y$ are \textit{isomorphic}, denoted by $X \cong Y$, if there exists a bijection $f \colon V(X) \to V(Y)$ such that both $f$ and $f^{-1}$ are simplicial maps. Furthermore, if $X$ and $Y$ have disjoint vertex sets, then their \textit{join} is the simplicial complex 
$X * Y = \{\sigma \cup \tau \mid \sigma \in X, \tau \in Y\}$. We write $X * X$ for the join of $X$ with a disjoint copy of itself, and $X^{*k}$ for the $k$-fold join $X * X * \cdots * X$ ($k$ times).
	
	Let $X$ be a finite $d$-dimensional simplicial complex, and let $R$ be a commutative ring
with identity. For $-1 \leq k \leq d$, let $C_k(X;R)$ denote the $k$-th chain group of $X$ with coefficients in $R$. We denote the boundary and coboundary maps respectively by 
	\[
	\partial_{k}(X;R):C_k(X;R)\rightarrow C_{k-1}(X;R)\text{ and }
	\partial_{k}^*(X;R):C_{k-1}(X;R)\rightarrow C_{k}(X;R),
	\]
where $\partial_{i}^*(X;R)$ is the adjoint operator of $\partial_{i}(X;R)$ with respect to the natural inner product. Since $\partial_k \circ \partial_{k+1} = 0$, the quotient module $\widetilde{H}_k(X;R)=\ker \partial_k / \operatorname{im} \partial_{k+1}$ is called  the \textit{$k$-th (reduced) homology group} of $X$ over $R$. When $R=\mathbb{R}$, the \textit{$k$-dimensional (reduced) Laplacian} of $X$ is  defined by
$$L_k(X)=\partial_{k+1}(X;\mathbb{R})\partial_{k+1}^*(X;\mathbb{R})+\partial_{k}^*(X;\mathbb{R})\partial_{k}(X;\mathbb{R}),$$
and its smallest eigenvalue,  denoted by $\mu_k(X)$, is called the \textit{$k$-th spectral gap} of $X$.

The combinatorial Laplacian, introduced by Eckmann \cite{Eck44} as a discrete analogue of the Hodge Laplacian, was originally defined for cellular complexes---a more general class than simplicial complexes. In this foundational work, Eckmann proved the discrete Hodge theorem. Below we present the simplicial version of this result, which has become the most widely used formulation in applications.

\begin{theorem}[Simplicial Hodge theorem]\label{thm::1.1}
    For a simplicial complex $X$, we have
    \[\widetilde{H}_k(X;\mathbb{R})\cong \ker L_k(X).\]
\end{theorem}

According to Theorem \ref{thm::1.1}, the vanishing of the $k$-th (reduced) homology group $\widetilde{H}_k(X;\mathbb{R})$ is equivalent to the positivity of the $k$-th spectral gap $\mu_k(X)$. This profound connection between combinatorial structures and topological invariants has motivated extensive research on combinatorial Laplacians for simplicial complexes. Over the past several decades, the spectra of combinatorial Laplacians have been investigated for various well-known families of simplicial complexes \cite{DW02,FH98,DR02,KRS02} as well as random complexes \cite{GW16,SY20}. A general framework for combinatorial Laplacians was established in \cite{HJ13}. Notably, several classical graph-theoretic results, including the matrix-tree theorem and Cheeger inequalities, have been successfully extended to simplicial complexes through combinatorial Laplacians; see \cite{Adi92,Kal83,DKM09,DKM11,DKM15} and \cite{GS15,JZ25,PRT16,SKM14}, respectively. For further developments on combinatorial Laplacians of simplicial complexes, we refer the reader to \cite{ABM05,Den01,HJ13+,Lew20,Lew20+,Lew23,Lim20,SBHLJ20,ZHL26}.

	Given a subcomplex $A$ of $X$, the pair $(X,A)$ is called a \textit{complex pair}. For a complex pair $(X,A)$, let $C_k(X,A;R)$ denote the quotient module $C_k(X;R)/C_k(A;R)$ for $-1 \leq k \leq d$. The boundary map $\partial_k(A;R): C_k(A;R) \rightarrow C_{k-1}(A;R)$ naturally induces a quotient map $\partial_k(X,A;R): C_k(X,A;R) \rightarrow C_{k-1}(X,A;R)$, called the \textit{relative boundary map}. The \textit{relative coboundary map} $\partial_k^*(X,A;R): C_{k-1}(X,A;R) \rightarrow C_{k}(X,A;R)$ is the adjoint operator of $\partial_k(X,A;R)$ with respect to the natural inner product. Since $\partial_k \circ \partial_{k+1} = 0$, the quotient module $H_k(X,A;R) = \ker \partial_k / \operatorname{im} \partial_{k+1}$ is called the \textit{$k$-th relative homology group} of $(X,A)$ over $R$.  The \textit{$k$-th relative Betti number} $\beta_k(X,A;R)$ is the rank of the largest free $R$-module summand of $H_k(X,A;R)$. When $R = \mathbb{Z}$, we simply write $H_k(X,A)$ and $\beta_k(X,A)$ instead of $H_k(X,A;\mathbb{Z})$ and $\beta_k(X,A;\mathbb{Z})$, respectively. When $R = \mathbb{R}$, the \textit{$k$-dimensional relative Laplacian} of the complex pair $(X,A)$ is defined by 
\[
L_k(X,A) = \partial_{k+1}(X,A;\mathbb{R}) \partial_{k+1}^*(X,A;\mathbb{R}) + \partial_{k}^*(X,A;\mathbb{R}) \partial_{k}(X,A;\mathbb{R}),
\] 
and its smallest eigenvalue, denoted by $\mu_k(X,A)$,  is called the \textit{$k$-th spectral gap} of $(X,A)$.  Furthermore, the \textit{up-down part} $\partial_{k+1}(X,A;\mathbb{R}) \partial_{k+1}^*(X,A;\mathbb{R})$ of $L_k(X,A)$ is denoted by $L^{\operatorname{ud}}_k(X,A)$. Note that if $A = \{\emptyset\}$, then $L_k(X,A) = L_k(X)$ for all $k \geq 1$.

The combinatorial Laplacians of complex pairs were first used by Duval \cite{Duv05} to provide an elegant expression for the error term in the recursion formula for the spectrum polynomial of a matroid. In the same work, Duval showed that the Laplacian eigenvalues of any pair of shifted complexes can be expressed in terms of the conjugate partition of the degree sequence of the shifted complex pair. Further results on combinatorial Laplacians of complex pairs can be found in \cite{Duv06,HJ13+}.

As is well-known, the combinatorial Laplacians and discrete Hodge theorem can be naturally extended to finite-dimensional chain complexes over $\mathbb{R}$ (see \cite[Proposition 2.1]{Fri98}). Since every complex pair naturally induces a chain complex (see Section~\ref{section::2} for details), we have the following relative version of the simplicial Hodge theorem.

\begin{theorem}[Relative simplicial Hodge theorem]\label{thm::1.2}
    For a complex pair $(X,A)$, we have
    \[ H_k(X,A;\mathbb{R}) \cong \ker L_k(X,A). \]
\end{theorem}

In this paper, we study combinatorial Laplacians for complex pairs, with our main contributions focusing on two aspects: establishing a relative version of the matrix-tree theorem and deriving lower bounds for the spectral gaps of complex pairs.

First, we present a relative version of the matrix-tree theorem for complex pairs. We begin with the definition of relative spanning forests (and trees) for complex pairs.

\begin{definition}\label{def::3.1}
\rm
    Let $(X,A)$ be a complex pair with $\dim X = d$, and let $\Upsilon$ be a simplicial complex such that $A_{(k)} \subseteq \Upsilon \subseteq X_{(k)}$, where $k \leq d$. We say $\Upsilon$ is a \textit{$k$-dimensional relative spanning forest} of $(X,A)$ if $\Upsilon_{(k-1)} = X_{(k-1)}$ and the following three conditions hold:
    \begin{enumerate}[(a)]\setlength{\itemsep}{0pt}
        \item $\beta_k(\Upsilon,A_{(k)}) = 0$,
        \item $\beta_{k-1}(\Upsilon,A_{(k)}) = \beta_{k-1}(X,A)$, and
        \item $f_k(\Upsilon,A_{(k)}) = f_k(X,A) - \beta_k(X_{(k)},A_{(k)})$.
    \end{enumerate}
  Moreover, we say $\Upsilon$ is a \textit{$k$-dimensional relative spanning tree} of $(X,A)$ if (b) is replaced by 
    \begin{enumerate}\setlength{\itemsep}{0pt}
        \item[(b$^*$)] $\beta_{k-1}(\Upsilon,A_{(k)}) = 0$.
    \end{enumerate}
The set of $k$-dimensional relative spanning forests (resp. trees) of the pair $(X,A)$ is denoted by $\mathcal{F}_k(X,A)$ (resp. $\mathcal{T}_k(X,A)$).
\end{definition}

\begin{remark}\rm
Note that when $k \geq 1$ and $A = \{\emptyset\}$, the notions of relative spanning trees and relative spanning forests coincide with the definitions of spanning trees and spanning forests for simplicial complexes in \cite{DKM09,DKM15}. Moreover, when $k = 1$, our definition agrees with that given in \cite{Chu96} for graph pairs.
\end{remark}

Under Definition \ref{def::3.1}, we obtain the following result that enumerates relative spanning trees weighted by the squares of the orders of their relative homology groups for complex pairs, expressed in terms of the eigenvalues of relative Laplacians.

\begin{theorem}[Relative matrix-tree theorem]\label{thm::main1}
    Let $(X,A)$ be a complex pair with $\dim X = d$. Then $\mathcal{F}_k(X,A) \neq \emptyset$ and $\mathcal{T}_k(X,A) \subseteq \mathcal{F}_k(X,A)$ for $0 \leq k \leq d$. Moreover, $\mathcal{T}_k(X,A) \neq \emptyset$ if and only if $\beta_{k-1}(X,A) = 0$. In this case, let $\beta = \beta_{k-2}(X,A)$, then:
    \begin{enumerate}\setlength{\itemsep}{0pt}
        \item[(i)] The product of all nonzero eigenvalues of $L_{k-1}^{\operatorname{ud}}(X,A)$ is equal to 
        \begin{equation*}
            \sum_{\Upsilon \in \mathcal{T}_k(X,A)} |H_{k-1}(\Upsilon,A_{(k)})|^2 \cdot \sum_{\Gamma \in \mathcal{F}_{k-1}(X,A)} \frac{|H_{k-2}(\Gamma,A_{(k-1)})/\mathbb{Z}^{\beta}|^2}{|H_{k-2}(X,A)/\mathbb{Z}^{\beta}|^2}.
        \end{equation*}
        \item[(ii)] For any $\Gamma \in \mathcal{F}_{k-1}(X,A)$,
        \begin{equation*}
            \sum_{\Upsilon \in \mathcal{T}_k(X,A)} |H_{k-1}(\Upsilon,A_{(k)})|^2 = \frac{|H_{k-2}(X,A)/\mathbb{Z}^{\beta}|^2}{|H_{k-2}(\Gamma,A_{(k-1)})/\mathbb{Z}^{\beta}|^2} \det L_{k-1}^{\operatorname{ud}}(X,\Gamma).
        \end{equation*}
    \end{enumerate}
\end{theorem}

\begin{remark}\rm
    When $k \geq 1$ and $A = \{\emptyset\}$, Theorem \ref{thm::main1} reduces to the matrix-tree theorem for simplicial complexes established by Duval, Klivans, and Martin \cite{DKM09} (see also \cite{DKM11,DKM15} for generalizations to cellular complexes). For the case $k = 1$, Theorem \ref{thm::main1} specializes to the matrix-tree theorem for Dirichlet eigenvalues (i.e., the eigenvalues of graph Laplacians with Dirichlet boundary conditions) for graph pairs \cite[Theorem 3]{Chu96}. Moreover, when $k = 1$ and $A = \{\emptyset\}$, Theorem \ref{thm::main1} recovers the classical matrix-tree theorem for graphs.
\end{remark}

Second, we derive several lower bounds for the spectral gaps of complex pairs and characterize the equality case for one sharp lower bound. As by-products, we obtain sufficient conditions for the vanishing of relative homology. 

Let $\Delta^m$ denote the complete simplicial complex on $m+1$ vertices.  In 2020, Lew \cite{Lew20} established a lower bound for the $k$-th spectral gap $\mu_k(X)$ in terms of the minimum degree of $\sigma\in X_k$ and the maximum dimension of a missing face of $X$, and proposed a conjecture regarding the unique simplicial complex that attains this lower bound. Very recently, Zhan, Huang, and Lin \cite{ZHL26} confirmed Lew's conjecture. The main result of these works is stated as follows.
	
	\begin{theorem}[{\cite[Theorem 1.1, Conjecture 5.1]{Lew20}, \cite[Theorem 1.3]{ZHL26}}]\label{thm::Lew}
		Let $X$ be a simplicial complex on the vertex set $V$ of size $n$ with $h(X) = h$. Then for all $k \geq -1$,
\begin{equation*}
    \mu_k(X) \geq (h + 1)\min_{\sigma \in X_k} \deg_X(\sigma) + (h+1)(k + 1) - hn \geq (h + 1)(k + 1) - hn.
\end{equation*}
Moreover, equality $\mu_k(X) = (h + 1)(k + 1) - hn$ holds if and only if
\begin{equation*}
    X \cong \left(\Delta^{h}_{(h - 1)}\right)^{*(n - k - 1)} * \Delta^{(h + 1)(k + 1) - hn - 1}.
\end{equation*}
	In this case, $\dim X=k$ and every eigenvector of $L_k(X)$ corresponding to $\mu_k(X)$ has no zero entries.
	\end{theorem}

Let $X$ be a finite $d$-dimensional simplicial complex. We define
\begin{equation*}
    B(X) = \{\sigma \in X : \sigma \subseteq \tau \text{ for some } \tau \in X_{d-1} \text{ with } \deg_{X}(\tau) \leq 1\}.
\end{equation*}
Clearly, $B(X)$ is a $(d-1)$-dimensional subcomplex of $X$. Note that if $X$ is a triangulation of a manifold $\mathcal{X}$ with boundary $\partial\mathcal{X}$, then $B(X)$ corresponds to $\partial\mathcal{X}$. This observation motivates the definition of a discrete boundary for simplicial complexes.

For any $k\leq d$, we say that a subcomplex $A \subseteq X$ is a \textit{$k$-th discrete boundary} of $X$ if it satisfies $\deg_X(\sigma) \leq 1$ for all $\sigma \in A_{k-1}$. Note that every $k$-th discrete boundary of $X$ has dimension at most $k$, and every subcomplex of $X$ with dimension at most $k-2$ is automatically a $k$-th discrete boundary of $X$. Furthermore, we observe that the subcomplex $B(X)$ is a $d$-th discrete boundary of $X$, and the $1$-th discrete boundaries of $X$ coincide with the boundaries of the Steklov eigenvalues of $G_X$ \cite{HH22,LZ25}. For a given $k$-th discrete boundary $A$ of $X$, we define
\begin{equation}\label{equ::26}
X' = \left\{\sigma \mid \sigma \subseteq \tau \text{ for some } \tau \in (X_k \setminus A_k) \cup \left(\cup_{i \geq k+1} X_i\right)\right\}.
\end{equation}
One can readily verify that $X'$ constitutes a well-defined subcomplex of $X$.  

Building upon  Theorem \ref{thm::Lew}, we obtain the following lower bound for the spectral gap $\mu_k(X,A)$.

	 \begin{theorem}\label{thm::main2}
	 Let $X$ be a simplicial complex, and let $A$ be a $k$-th discrete boundary of $X$.  Let $X'$ be the subcomplex of $X$ defined by \eqref{equ::26}. Then, for any $k\geq 1$, 
\begin{equation*}
\begin{aligned}
\mu_k(X, A) &\geq \min_{\sigma\in X'_k}\left((h'+1)\deg_{X'}(\sigma)-|\sigma_{k-1}\cap A_{k-1}|\right)+(h'+1)(k+1)-h'n'\\
&\geq (h' + 1)(k + 1) - h'n' - \max_{\sigma \in X'_k} |\sigma_{k-1} \cap A_{k-1}|,
\end{aligned}
\end{equation*}
where $h' = h(X')$ and $n'$ is the number of vertices in $X'$. Moreover, equality $\mu_k(X, A) = (h' + 1)(k + 1) - h'n' - \max_{\sigma \in X'_k} |\sigma_{k-1} \cap A_{k-1}|$ holds if and only if 
\begin{equation*}
X' \cong \Big(\Delta^{h'}_{(h' - 1)}\Big)^{*(n' - k - 1)} * \Delta^{(h' + 1)(k + 1) - h'n' - 1}
\end{equation*}
and  $|\sigma_{k-1} \cap A_{k-1}| = c$ (a constant) for all $\sigma \in X'_k$.
\end{theorem}

Recall that $G_X$ is the $1$-skeleton of $X$, viewed as a graph. For $k=0$, one can verify that $\mu_0(X)$ coincides with the second smallest eigenvalue $\lambda_2(G_X)$ of the graph Laplacian $L(G_X)$.

Let $G$ be a graph with vertex set $V(G)$. The \textit{flag complex} of $G$, denoted by $X(G)$, is the simplicial complex on $V(G)$ whose simplices are the subsets of $V(G)$ that span complete subgraphs of $G$. Note that $G_{X(G)} = G$.

For a flag complex pair, we establish the following lower bound for the spectral gap $\mu_k(X,A)$.

\begin{theorem}\label{thm::main3}
			Let $X$ be a flag complex, and  let $A$ be a $k$-th discrete boundary of $X$. Let $X'$ be the subcomplex of $X$ defined by \eqref{equ::26}. Then, for any $k\geq 1$, 
			\begin{equation*}
				\mu_k(X,A)\geq(k+1)\lambda_2(G_{X'})-kn'-\max_{\sigma\in X'_k}|\sigma_{k-1}\cap A_{k-1}|,
			\end{equation*}
where $n'$ is the number of vertices in $X'$. Consequently, if 
$$\lambda_2(G_{X'})>\frac{1}{k+1}\left(kn'+\max_{\sigma\in X'_k}|\sigma_{k-1}\cap A_{k-1}|\right),$$ then $H_k(X,A;\mathbb{R})=0$.
	\end{theorem}
	
For a general complex pair $(X,A)$, we have the following inequality relation between $\mu_k(X,A)$ and $\mu_k(X)$.

\begin{theorem}\label{thm::main4}
		Let $(X,A)$ be a complex pair. For  $k\geq 1$, we have
		\begin{equation*}
			\mu_k(X,A)\geq \mu_k(X)-\max_{\sigma\in X_k\setminus A_k}\left(|\sigma_{k-1}\cap A_{k-1}|+|\{\tau\in X_k\setminus A_k:\tau\cap\sigma\in A_{k-1}\}|\right).
		\end{equation*}
Consequently, if 
\begin{equation*}
\mu_k(X)>\max_{\sigma\in X_k\setminus A_k}\left(|\sigma_{k-1}\cap A_{k-1}|+|\{\tau\in X_k\setminus A_k:\tau\cap\sigma\in A_{k-1}\}|\right),
		\end{equation*}
		then $H_k(X,A;\mathbb{R})=0$.
	\end{theorem}
	
	Using Garland's ``local to global'' method (see \cite[Theorem 5.9]{Gar73}), Aharoni, Berger, and Meshulam \cite{ABM05} established an inequality concerning the spectral gaps of flag complexes.
	
	 \begin{theorem}[{\cite[Theorem 1.1]{ABM05}}]\label{thm::ABM}
Let $X$ be a flag complex on $n$ vertices. For $k\geq 1$,
\[
k\mu_k(X)\geq (k+1)\mu_{k-1}(X)-n.
\] 
 \end{theorem} 
	
From Theorems \ref{thm::1.2},  \ref{thm::main4} and  \ref{thm::ABM}, we immediately obtain a sufficient condition for the vanishing of relative homology of flag complex pairs.
	\begin{corollary} 
	Let $X$ be a flag complex on $n$ vertices, and let $A$ be a subcomplex of $X$.  For $k\geq 1$, if
	\begin{equation*}
			\lambda_2(G_X)>\frac{1}{k+1}\left(kn+\max_{\sigma\in X_k\setminus A_k}\left[|\sigma_{k-1}\cap A_{k-1}|+|\{\tau\in X_k\setminus A_k:\tau\cap\sigma\in A_{k-1}\}|\right]\right),
		\end{equation*}
		then $H_k(X,A;\mathbb{R})=0$.
	\end{corollary}

By employing discrete boundary and relative homology, we further derive the following estimate for the spectral gap of a pure simplicial complex.

	\begin{theorem}\label{thm::main5}
		Let $X$ be a pure $d$-dimensional simplicial complex. If $H_d(X,B(X);\mathbb{R})\neq 0$, then 
		\begin{equation*}
			\min_{\sigma\in X_d}|\sigma_{d-1}\cap B(X)_{d-1}|\leq \mu_d(X)\leq \max_{\sigma\in X_d}|\sigma_{d-1}\cap B(X)_{d-1}|.
		\end{equation*}
	\end{theorem}

The paper is organized as follows. In Section \ref{section::2}, we review necessary notations from algebraic topology and discuss fundamental properties of relative Laplacians for complex pairs. In Section \ref{section::3}, we present some basic properties concerning relative spanning forests and relative spanning trees, which are then used to prove Theorem \ref{thm::main1}. Finally, in Section \ref{section::4}, we prove Theorems \ref{thm::main2}--\ref{thm::main4} and \ref{thm::main5} in sequence. Moreover, we construct an explicit example of a pure $d$-dimensional simplicial complex satisfying the condition $H_d(X,B(X);\mathbb{R})\neq 0$, as addressed in Theorem \ref{thm::main5}. As a direct corollary of Theorem \ref{thm::main5}, we also obtain estimates for the $d$-th spectral gaps of $d$-paths, (orientable) $d$-cycles, and $d$-stars (see also \cite[Section 4]{HJ13}).

	\section{Preliminaries}\label{section::2}
 In this section, we review some notations from algebraic topology and discuss fundamental properties of relative Laplacians for complex pairs.

 Let $R$ be a commutative ring with identity. A \textit{chain complex} $(C_{\bullet}, \partial_{\bullet})$ is a sequence of $R$-modules $\dots, C_0, C_1, C_2, \dots$ connected by homomorphisms (called \textit{boundary operators}) $\partial_i: C_i \to C_{i-1}$, such that $\partial_i \circ \partial_{i+1} = 0$ for all $i$. The complex may be represented as:
\[
\cdots \longrightarrow C_{i+1} \overset{\partial_{i+1}}{\longrightarrow} C_i \overset{\partial_i}{\longrightarrow} C_{i-1} \longrightarrow \cdots.
\]

Let $X$ be a finite $d$-dimensional simplicial complex. We say that a face $\sigma \in X$ is \textit{oriented} if we choose an ordering of its vertices. Two orderings of the vertices are said to determine \textit{the same orientation} if there exists an even permutation that transforms one ordering into the other, and \textit{opposite orientations} if the permutation is odd.

Let $\sigma \in X$ be a face with vertices ordered by $\prec: v_0 \prec v_1 \prec \cdots \prec v_k$. We denote the \textit{oriented $k$-face} by $[\sigma_\prec] = [v_0, \ldots, v_k]$. For any other ordering $\prec'$ of the vertices of $\sigma$, we define $[\sigma_{\prec'}] = [\sigma_{\prec}]$ if $\prec'$ and $\prec$ determine the same orientation, and $[\sigma_{\prec'}] = -[\sigma_{\prec}]$ if they determine opposite orientations. Moreover, if $\eta = \sigma \setminus \{v_i\} = \{v_0, \ldots, v_{i-1}, v_{i+1}, \ldots, v_k\}$, we define
\[
\operatorname{sgn}([\eta_\prec], [\sigma_\prec]) := (-1)^i \text{ and } \operatorname{sgn}([\eta_\prec], -[\sigma_\prec]) := -(-1)^{i}.
\]

Throughout this paper, unless otherwise stated, we always assume that $\prec$ is a fixed total order on the vertex set $V(X)$ of $X$. This total order induces a canonical orientation on each face $\sigma \in X$, and we simply write $[\sigma] = [\sigma_\prec]$.

For $-1 \leq k \leq d$, let $C_k(X; R)$ be the $k$-th chain group of $X$ with coefficients in $R$, i.e., the free $R$-module with basis $[X_k] := \{[\sigma] \mid \sigma \in X_k\}$. The \textit{boundary map} $\partial_k(X; R):C_k(X; R)\rightarrow C_{k-1}(X; R)$  is the $R$-module homomorphism defined by $R$-linearly extending the following map on basis elements:
\begin{equation}\label{equ::1}
    \partial_k(X; R)[\sigma] = \sum_{\substack{\eta\in X_{k-1}\\\eta\subseteq \sigma}} \operatorname{sgn}([\eta],[\sigma]) \cdot [\eta], \quad \text{for any } \sigma\in X_k.
\end{equation}
The \textit{coboundary map} $\partial_k^*(X; R):C_{k-1}(X; R)\rightarrow C_{k}(X; R)$ is the adjoint operator of $\partial_k(X; R)$ with respect to the natural inner product.  We see that  $\partial_k^*(X; R)$ is given by
\begin{equation*}\label{equ::coboundary_map}
    \partial_k^*(X; R)[\sigma] =  \sum_{\substack{\eta\in X_{k}\\\sigma\subseteq \eta}}\operatorname{sgn}([\sigma],[\eta]) \cdot [\eta], \quad \text{for any } \sigma\in X_{k-1}.
\end{equation*}
For notational convenience, we will omit $(X;R)$ in $\partial_k(X; R)$ and $\partial_k^*(X; R)$ when no ambiguity arises.  One can verify that $\partial_{k}\circ \partial_{k+1}=0$ for all $0 \leq k < d$. Thus, $X$ defines a chain complex:
	\[
			0 \longrightarrow C_d(X;R)\overset{\partial_d}{\longrightarrow} C_{d-1}(X;R) \overset{\partial_{d-1}}{\longrightarrow} \cdots \overset{\partial_0}{\longrightarrow}   C_{-1}(X;R)  \longrightarrow 0.
	\]
The quotient module $\widetilde{H}_k(X;R) := \ker \partial_k / \operatorname{im} \partial_{k+1}$ is called the \textit{$k$-th reduced homology group} of $X$ over $R$. When $R = \mathbb{Z}$, we simplify notation to $C_k(X)$ and $\widetilde{H}_k(X)$. The \textit{$k$-th reduced Betti number} $\widetilde{\beta}_k(X)$ is the rank of the largest free $\mathbb{Z}$-module summand of $\widetilde{H}_k(X)$.

The \textit{$k$-dimensional Laplacian} of $X$ is an $\mathbb{R}$-module homomorphism $L_k(X):C_k(X;\mathbb{R})\rightarrow C_k(X;\mathbb{R})$ defined by
\[
L_k(X) = \partial_{k+1}(X;\mathbb{R})\partial_{k+1}^*(X;\mathbb{R}) + \partial_{k}^*(X;\mathbb{R})\partial_{k}(X;\mathbb{R}).
\]
Since $C_k(X;\mathbb{R})$ is a finite-dimensional real vector space with basis $[X_k]$, we can regard $\partial_k(X;\mathbb{R})$ as a matrix with rows indexed by $X_{k-1}$ and columns by $X_k$, and $\partial_{k}^*(X;\mathbb{R})$ as the transpose matrix $\partial_{k}(X;\mathbb{R})^\top$. Thus, $L_k(X)$ can be represented by a positive semidefinite matrix  with rows and columns indexed by $X_k$.

For $\sigma,\tau \in X_k$ satisfying $|\sigma \cap \tau| = k$, we have $(\sigma \setminus \tau) \cup (\tau \setminus \sigma) = \{u,v\}$ for some $u \prec v \in V$. Let $\epsilon(\sigma,\tau)$ denote the size of the set $\{w \in \sigma \cap \tau : u \prec w \prec v\}$. Then the matrix representation of $L_k(X)$ is given by:

\begin{lemma}[{\cite{DR02,Gol12}}]\label{lem::2.1}
		Let $k\geq 0$. Then, for any $\sigma,\tau\in X_k$,
		\begin{equation*}
			L_k(X)(\sigma,\tau)=\begin{cases} 
			\deg_X(\sigma)+k+1 & \mbox{if $\sigma=\tau$}, \\
				(-1)^{\epsilon(\sigma,\tau)} &\mbox{if $\sigma\cup\tau\notin X_{k+1},~\sigma\cap\tau\in X_{k-1}$}, \\
				0 & \mbox{otherwise}.
			\end{cases}   
		\end{equation*}
	\end{lemma}

Let $A$ be a subcomplex of $X$. For the complex pair $(X,A)$, we denote by $C_k(X,A;R)$ the quotient module $C_k(X;R)/C_k(A;R)$ for all $-1 \leq k \leq d$. By definition, we observe that $C_{-1}(X,A;R)= 0$. Since the boundary map $\partial_k(A;R): C_k(A;R) \to C_{k-1}(A;R)$ is the restriction of the boundary map $\partial_k(X;R)$ on $C_k(A;R)$, the boundary map $\partial_k(X;R): C_k(X;R) \to C_{k-1}(X;R)$ naturally induces a quotient map $\partial_k(X,A;R): C_k(X,A;R) \to C_{k-1}(X,A;R)$, called the \textit{relative boundary map}. The \textit{relative coboundary map} $\partial_k^*(X,A; R): C_{k-1}(X,A; R) \to C_{k}(X,A; R)$ is the adjoint operator of $\partial_k(X,A; R)$ with respect to the natural inner product. For convenience, we will omit $(X,A;R)$ in these notations whenever no ambiguity arises.  

The relative boundary maps satisfy $\partial_k \circ \partial_{k+1} = 0$ for all $k$, and thus the complex pair $(X,A)$ defines a chain complex:
\[
0 \longrightarrow C_{d}(X,A;R) \overset{\partial_{d}}{\longrightarrow} C_{d-1}(X,A;R) \overset{\partial_{d-1}}{\longrightarrow} \cdots \overset{\partial_1}{\longrightarrow} C_{0}(X,A;R) \overset{\partial_0}{\longrightarrow} C_{-1}(X,A;R) = 0.
\]
The quotient module $H_k(X,A;R) = \ker \partial_k / \operatorname{im} \partial_{k+1}$ is called the \textit{$k$-th relative homology group} of $(X,A)$ over $R$. The \textit{$k$-th relative Betti number} $\beta_k(X,A;R)$ is defined as the rank of the largest free $R$-module summand of $H_k(X,A;R)$. When $R = \mathbb{Z}$, we simply write $C_k(X,A)$, $\partial_k(X,A)$, $H_k(X,A)$, and $\beta_k(X,A)$ instead of $C_k(X,A;\mathbb{Z})$, $\partial_k(X,A;\mathbb{Z})$, $H_k(X,A;\mathbb{Z})$, and $\beta_k(X,A;\mathbb{Z})$, respectively.  For $R = \mathbb{R}$, the universal coefficient theorem (see, for example, \cite[Section 3.A]{Hat02}) implies that
\begin{equation}\label{equ::2} 
H_k(X,A;\mathbb{R}) \cong H_k(X,A) \otimes_{\mathbb{Z}} \mathbb{R}, 
\end{equation}
since $\mathbb{R}$ is a flat $\mathbb{Z}$-module. Thus, $\beta_k(X,A;\mathbb{R}) = \beta_k(X,A)$, which allows us to use $\beta_k(X,A)$ in the case $R = \mathbb{R}$ as well.

Note that $C_k(X,A;R)$ can be regarded as a free $R$-submodule of $C_k(X;R)$ with basis $[X_k \setminus A_k] = \{[\sigma] \mid \sigma \in X_k \setminus A_k\}$. Therefore, the relative boundary map $\partial_k(X,A;R)$ can be defined as the $R$-module homomorphism from $C_k(X,A;R)$ (viewed as a free $R$-submodule of $C_k(X;R)$ with basis $[X_k \setminus A_k]$) to $C_{k-1}(X,A;R)$ (viewed as a free $R$-submodule of $C_{k-1}(X;R)$ with basis $[X_{k-1} \setminus A_{k-1}]$) given by
\begin{equation}\label{equ::3}
    \partial_k(X,A;R)([\sigma]) = \sum_{\substack{\eta \in X_{k-1} \setminus A_{k-1} \\ \eta \subseteq \sigma}} \operatorname{sgn}([\eta],[\sigma]) \cdot [\eta], \quad \text{for any } \sigma \in X_k \setminus A_k.
\end{equation}
The relative coboundary map $\partial_k^*(X,A;R): C_{k-1}(X,A;R) \to C_{k}(X,A;R)$ is then given by
\begin{equation}\label{equ::4}
    \partial_k^*(X,A; R)[\sigma] = \sum_{\substack{\eta \in X_{k} \setminus A_k \\ \sigma \subseteq \eta}} \operatorname{sgn}([\sigma],[\eta]) \cdot [\eta], \quad \text{for any } \sigma \in X_{k-1} \setminus A_{k-1}.
\end{equation}

The \textit{$k$-dimensional relative Laplacian} of the complex pair $(X,A)$ is an $\mathbb{R}$-module homomorphism from $C_k(X,A;\mathbb{R})$ to itself, defined by
\[
L_k(X,A) = \partial_{k+1}(X,A;\mathbb{R}) \partial_{k+1}^*(X,A;\mathbb{R}) + \partial_{k}^*(X,A;\mathbb{R}) \partial_{k}(X,A;\mathbb{R}).
\]
The operator $L_k(X,A)$ decomposes into the \textit{up-down part} $L^{\operatorname{ud}}_k(X,A):=\partial_{k+1}(X,A;\mathbb{R})\partial_{k+1}^*(X,A;\mathbb{R})$ and the  \textit{down-up part} $L^{\operatorname{du}}_k(X,A):=\partial_{k}^*(X,A;\mathbb{R})\partial_{k}(X,A;\mathbb{R})$. 

According to the above definitions, we first provide a detailed proof of Theorem \ref{thm::1.2} for the sake of completeness.
	
\renewcommand\proofname{\bf{Proof of Theorem \ref{thm::1.2}}}
	\begin{proof}
		Since $\partial_k\partial_{k+1}=0$ and $\partial^*_{k+1}\partial^*_k=0$, we have
		\begin{equation}\label{equ::5}
			\operatorname{im} L_k^{\operatorname{du}}(X,A)\subseteq\ker  L_k^{\operatorname{ud}}(X,A),
		\end{equation}
		\begin{equation}\label{equ::6}
			\operatorname{im} L_k^{\operatorname{ud}}(X,A)\subseteq\ker  L_k^{\operatorname{du}}(X,A).
		\end{equation}
		Therefore, 
		\begin{equation*}
			\begin{aligned}
				\ker  L_k(X,A) &=\ker  \partial_{k+1}\partial^*_{k+1}\cap\ker \partial^*_k\partial_k\\
				&=\ker \partial^*_{k+1}\cap\ker \partial_k\\
				&=(\operatorname{im}\partial_{k+1})^{\bot}\cap\ker \partial_k\\
				&\cong H_k(X,A;\mathbb{R}),
			\end{aligned}
		\end{equation*}
		and the result follows.
	\end{proof}

Since the $\mathbb{R}$-module $C_k(X,A;\mathbb{R})$ can be viewed as a real vector space with basis $[X_k \setminus A_k]$, we can regard $\partial_k(X,A;\mathbb{R})$ as a matrix with rows indexed by $X_{k-1} \setminus A_{k-1}$ and columns indexed by $X_k \setminus A_k$, and  $\partial_{k}^*(X,A;\mathbb{R})$ as the transpose matrix $\partial_{k}(X,A;\mathbb{R})^\top$. Thus, $L_k(X,A)$ can be represented by a positive semi-definite matrix with rows and columns indexed by $X_k \setminus A_k$. Below, we present the explicit matrix representation of $L_k(X,A)$.

	\begin{lemma}\label{lem::2.2} 
		Let $\sigma,\tau\in X_k\setminus A_k$. The $(\sigma,\tau)$-entries of $L^{\operatorname{ud}}_k(X,A)$ and $L^{\operatorname{du}}_k(X,A)$ are respectively given by 
		\begin{equation*}
			L^{\operatorname{ud}}_k(X,A)(\sigma,\tau)=\begin{cases} 
				\deg_X(\sigma) & \mbox{if $\sigma=\tau$}, \\
				-(-1)^{\epsilon(\sigma,\tau)} &\mbox{if $\sigma\cup\tau\in X_{k+1}$}, \\
				0 & \mbox{otherwise},
			\end{cases}   
		\end{equation*}
		and 
		\begin{equation*}
			L^{\operatorname{du}}_k(X,A)(\sigma,\tau)=\begin{cases} 
				|\sigma_{k-1}\setminus A_{k-1}| & \mbox{if $\sigma=\tau$}, \\
				(-1)^{\epsilon(\sigma,\tau)} &\mbox{if $\sigma\cap\tau\in X_{k-1}\setminus A_{k-1}$}, \\
				0 & \mbox{otherwise}.
			\end{cases}   
		\end{equation*}
Consequently, the $(\sigma,\tau)$-entry of $L_k(X,A)$ is  
		\begin{equation*}
			L_k(X,A)(\sigma,\tau)=\begin{cases} 
				\deg_X(\sigma)+|\sigma_{k-1}\setminus A_{k-1}| & \mbox{if $\sigma=\tau$}, \\
				(-1)^{\epsilon(\sigma,\tau)} &\mbox{if $\sigma\cup\tau\notin X_{k+1},~\sigma\cap\tau\in X_{k-1}\setminus A_{k-1}$}, \\
				-(-1)^{\epsilon(\sigma,\tau)} &\mbox{if $\sigma\cup\tau\in X_{k+1},~\sigma\cap\tau\in A_{k-1}$}, \\
				0 & \mbox{otherwise}.
			\end{cases}
		\end{equation*}
	\end{lemma}
\renewcommand\proofname{\bf{Proof}}
	\begin{proof} Let $\sigma\in X_k\setminus A_k$. Note that if $\eta\in X_{k+1}$ and $\sigma\subseteq \eta$, then   $\eta\notin A_{k+1}$. According to \eqref{equ::3} and  \eqref{equ::4},  we obtain
	\begin{align*}
				\partial_{k+1}\partial^*_{k+1}([\sigma])&=\partial_{k+1}\Bigg(\sum_{\substack{\eta\in X_{k+1}\\ \sigma\subseteq \eta}}\operatorname{sgn}([\sigma],[\eta])\cdot [\eta]\Bigg)\\
				&=\sum_{\substack{\eta\in X_{k+1}\\ \sigma\subseteq \eta}}\operatorname{sgn}([\sigma],[\eta])\sum_{\substack{\tau\in X_{k}\setminus A_k\\ \tau\subseteq \eta}}\operatorname{sgn}([\tau],[\eta])\cdot [\tau]\\
				&=\sum_{\substack{\eta\in X_{k+1}\\ \sigma\subseteq \eta}}\Bigg([\sigma]+\sum_{\substack{\tau\in X_{k}\setminus A_k\\ \sigma\neq \tau\subseteq \eta}}\operatorname{sgn}([\sigma],[\eta])\cdot \operatorname{sgn}([\tau],[\eta])\cdot [\tau]\Bigg)\\
				&=\sum_{\substack{\eta\in X_{k+1}\\ \sigma\subseteq \eta}}\Bigg([\sigma]+\sum_{\substack{\tau\in X_{k}\setminus A_k\\ \tau\cup \sigma=\eta}}(-1)^{\epsilon(\sigma,\tau)+1} [\tau]\Bigg)\\
				&=\operatorname{deg}_X(\sigma)[\sigma]+\sum_{\substack{\tau\in X_{k}\setminus A_k\\ \tau\cup \sigma\in X_{k+1}}}(-1)^{\epsilon(\sigma,\tau)+1} [\tau]
		\end{align*}
and 	
		\begin{align*}
				\partial_k^*\partial_k([\sigma])&=\partial_k^*\Bigg(\sum_{\substack{\eta\in X_{k-1}\setminus A_{k-1}\\\eta\subseteq \sigma}}\operatorname{sgn}([\eta],[\sigma])[\eta]\Bigg)\\
				&=\sum_{\substack{\eta\in X_{k-1}\setminus A_{k-1}\\ \eta\subseteq \sigma}}\operatorname{sgn}([\eta],[\sigma])\sum_{\substack{\tau\in X_{k}\setminus A_{k}\\ \eta \subseteq \tau}}\operatorname{sgn}([\eta],[\tau])\cdot [\tau]\\
				&=\sum_{\substack{\eta\in X_{k-1}\setminus A_{k-1}\\ \eta\subseteq \sigma}}\Bigg([\sigma]+\sum_{\substack{\tau\in X_{k}\setminus A_{k}\\ \eta \subseteq \tau\neq \sigma}}\operatorname{sgn}([\eta],[\sigma])\cdot \operatorname{sgn}([\eta],[\tau])\cdot [\tau]\Bigg)\\
				&=\sum_{\substack{\eta\in X_{k-1}\setminus A_{k-1}\\ \eta\subseteq \sigma}}\Bigg([\sigma]+\sum_{\substack{\tau\in X_{k}\setminus A_{k}\\ \tau\cap \sigma=\eta}}(-1)^{\epsilon(\sigma,\tau)} [\tau]\Bigg)\\
				&=|\sigma_{k-1}\setminus A_{k-1}|\cdot[\sigma]+\sum_{\substack{\tau\in X_k\setminus A_k\\ \tau\cap\sigma\in X_{k-1}\setminus A_{k-1}}}(-1)^{\epsilon(\sigma,\tau)}[\tau].
	        \end{align*}
By the definitions of $L^{\operatorname{ud}}_k(X,A)$, $L^{\operatorname{du}}_k(X,A)$ and $L_k(X,A)$, the result follows.
	\end{proof}
	
	\begin{remark}
	\rm It is worth mentioning that when $k \geq 1$ and $A = \{\emptyset\}$, the matrix $L_k(X, A)$ in Lemma \ref{lem::2.2} coincides with the matrix $L_k(X)$ given in Lemma \ref{lem::2.1}. Furthermore, when $k=0$, the matrix $L_k(X, A)$ becomes the graph Laplacian with Dirichlet boundary conditions for the graph pair $(G_X,A_X)$ (see \cite{Chu96} for details). In particular, when $k = 0$ and $A = \{\emptyset\}$, the matrix $L_k(X, A)$ reduces to the standard graph Laplacian $L(G_X)$. 
	\end{remark}
	
	In view of \eqref{equ::1} and \eqref{equ::3},  we have the following observation, which will be useful in subsequent analysis.
	\begin{lemma}\label{lem::2.3}
		Let $(X,A)$ be a complex pair. Then $\partial_k(X,A;\mathbb{R})$ is the submatrix of $\partial_k(X;\mathbb{R})$ with rows indexed by $X_{k-1}\setminus A_{k-1}$ and columns indexed by $X_k\setminus A_k$.
	\end{lemma}

	For any real symmetric matrix $M$ of order $m+1$, let $\mathbf{s}(M) = [\lambda_0, \ldots, \lambda_m]$ denote the weakly increasing sequence of its eigenvalues. In particular, we denote by $\lambda_{\max}(M)=\lambda_m$ and $\lambda_{\min}(M)=\lambda_0$. Furthermore, we write $\mathbf{s}(M) \stackrel{\circ}{=} \mathbf{s}(N)$ if the eigenvalue multisets of $M$ and $N$ differ only in the multiplicities of their zero eigenvalues. The disjoint union of multisets is denoted by $\sqcup$. 
    
According to \eqref{equ::5} and \eqref{equ::6}, we see that $\lambda$ is a nonzero eigenvalue of $L_k(X,A)$ if and only if it is an eigenvalue of $L^{\operatorname{ud}}_k(X,A)$ or $L^{\operatorname{du}}_k(X,A)$. Therefore,
\begin{equation*}
    \mathbf{s}(L_k(X,A)) \stackrel{\circ}{=} \mathbf{s}(L^{\operatorname{ud}}_k(X,A)) \sqcup \mathbf{s}(L^{\operatorname{du}}_k(X,A)).
\end{equation*}
As a direct consequence of the fact that $\mathbf{s}(MN) \stackrel{\circ}{=} \mathbf{s}(NM)$, for suitable real matrices $M$ and $N$, we obtain the following equality:
\begin{equation*}\label{equ::ud=du}
    \mathbf{s}(L^{\operatorname{ud}}_k(X,A)) \stackrel{\circ}{=} \mathbf{s}(L^{\operatorname{du}}_{k+1}(X,A)).
\end{equation*}

Let $f_k(X,A)$ denote the size of $X_k \setminus A_k$, and define
\begin{equation}\label{equ::7}
    \chi_{k-1}(X,A) = \sum_{j \geq k} (-1)^{j-k} (f_j(X,A) - \beta_j(X,A)), \quad \text{for } k \geq 0.
\end{equation}
Then the multiplicities of zero eigenvalues in $\mathbf{s}(L^{\operatorname{ud}}_k(X,A))$ and $\mathbf{s}(L^{\operatorname{du}}_k(X,A))$ are given as follows.

	\begin{lemma}\label{lem::2.4}
    The multiplicity of zero eigenvalues in $\boldsymbol{\operatorname{s}}(L^{\operatorname{ud}}_k(X,A))$ equals $f_k(X,A)-\chi_k(X,A)$, while in $\boldsymbol{\operatorname{s}}(L^{\operatorname{du}}_k(X,A))$ it equals $f_k(X,A)-\chi_{k-1}(X,A)$.
\end{lemma}

\begin{proof}
    Consider the following two short exact sequences of $\mathbb{R}$-modules:
    \begin{equation*}
    \begin{aligned}
        &0\longrightarrow \ker \partial_k \longrightarrow C_k(X,A;\mathbb{R}) \longrightarrow \operatorname{im}\partial_k \longrightarrow 0, \\
        &0\longrightarrow \operatorname{im}\partial_{k+1} \longrightarrow \ker \partial_k \longrightarrow H_k(X,A;\mathbb{R}) \longrightarrow 0.
    \end{aligned}
    \end{equation*}
    Since every $\mathbb{R}$-module is projective, these short exact sequences split. Thus we obtain
    \begin{equation*}
        f_k(X,A) = \dim(\ker \partial_k) + \dim(\operatorname{im}\partial_k), \quad 
        \dim(\ker \partial_k) = \dim(\operatorname{im}\partial_{k+1}) + \beta_k(X,A).
    \end{equation*}
    Combining this with \eqref{equ::7} yields that
    \begin{equation}\label{equ::8}
        \dim (\operatorname{im}\partial_k) = \chi_{k-1}(X,A).
    \end{equation}
Therefore, the multiplicity of zero eigenvalues in $\boldsymbol{\operatorname{s}}(L^{\operatorname{ud}}_k(X,A))$ is
    \begin{equation*}
        \dim (\ker L^{\operatorname{ud}}_k(X,A)) = \dim (\ker \partial_{k+1}^*) = f_k(X,A) - \dim(\operatorname{im}\partial_{k+1}) = f_k(X,A) - \chi_k(X,A),
    \end{equation*}
    while in $\boldsymbol{\operatorname{s}}(L^{\operatorname{du}}_k(X,A))$ it is
    \begin{equation*}
        \dim (\ker L^{\operatorname{du}}_k(X,A)) = \dim (\ker \partial_{k}) = f_k(X,A) - \dim(\operatorname{im}\partial_{k}) = f_k(X,A) - \chi_{k-1}(X,A).
    \end{equation*}
    
    This completes the proof.
\end{proof}

The well-known Euler-Poincar\'{e} formula (cf. \cite{Poi93,Poi99}) establishes a linear relation between the $f$-vectors and the reduced Betti numbers of simplicial complexes. We conclude this section by proving a relative version of the Euler-Poincar\'{e} formula, which plays a crucial role in the proof of Theorem \ref{thm::main1}. 

\begin{lemma}[Relative Euler-Poincar\'{e} formula]\label{lem::2.5}
    Let $(X,A)$ be a complex pair. Then 
    \begin{equation*}
        \sum_{i\geq 0}(-1)^i f_i(X,A) = \sum_{i\geq 0}(-1)^i \beta_i(X,A).
    \end{equation*}
\end{lemma}

\begin{proof}
    Since $\partial_0(X,A,\mathbb{R})$ is a linear transformation from $C_0(X,A,\mathbb{R})$ to $C_{-1}(X,A,\mathbb{R}) = 0$, we have
    \[
    \dim (\operatorname{im} \partial_0(X,A,\mathbb{R}))= 0.
    \] 
    Combining this with \eqref{equ::8} yields
    \[
    \sum_{i\geq 0}(-1)^i \big(f_i(X,A) - \beta_i(X,A)\big)= \chi_{-1} =\dim (\operatorname{im} \partial_0(X,A,\mathbb{R}))= 0,
    \]
    as desired.
\end{proof}

	\section{Proof of Theorem \ref{thm::main1}}\label{section::3}

In this section, we shall prove Theorem \ref{thm::main1}. To achieve this goal, we need a series of  lemmas.
	
	\begin{lemma}\label{lem::3.1}
    Let $(X,A)$ be a complex pair with $\dim X=d$, and let $\Upsilon$ be a simplicial complex such that $A_{(k)}\subseteq\Upsilon\subseteq X_{(k)}$ and $\Upsilon_{(k-1)}=X_{(k-1)}$. Then any two of the conditions (a), (b), (c) in Definition \ref{def::3.1} imply the third.
\end{lemma}

\begin{proof}
    By Lemma \ref{lem::2.5}, we have
    \begin{equation}\label{equ::9}
        \sum_{i=0}^k(-1)^if_i(\Upsilon,A_{(k)})=\sum_{i=0}^k(-1)^i\beta_i(\Upsilon,A_{(k)})
    \end{equation}
    and 
    \begin{equation}\label{equ::10}
        \sum_{i=0}^k(-1)^if_i(X_{(k)},A_{(k)})=\sum_{i=0}^k(-1)^i\beta_i(X_{(k)},A_{(k)}).
    \end{equation}
Since $\Upsilon_{(k-1)}=X_{(k-1)}$, it follows that $f_i(\Upsilon,A_{(k)})=f_i(X_{(k)},A_{(k)})$ for $i\leq k-1$ and $\beta_i(\Upsilon,A_{(k)})=\beta_i(X_{(k)},A_{(k)})$ for $i\leq k-2$. Combining this with \eqref{equ::9} and \eqref{equ::10}, we obtain
    \begin{equation*}
        f_k(\Upsilon,A_{(k)})-f_k(X_{(k)},A_{(k)})=(\beta_k(\Upsilon,A_{(k)})-\beta_{k-1}(\Upsilon,A_{(k)}))-(\beta_k(X_{(k)},A_{(k)})-\beta_{k-1}(X_{(k)},A_{(k)})),
    \end{equation*}
    which can be rewritten as
    \begin{equation*}
        f_k(\Upsilon,A_{(k)})-f_k(X_{(k)},A_{(k)})+\beta_k(X_{(k)},A_{(k)})=\beta_k(\Upsilon,A_{(k)})-(\beta_{k-1}(\Upsilon,A_{(k)})-\beta_{k-1}(X_{(k)},A_{(k)})).
    \end{equation*} 
    Note that $f_{k}(X_{(k)},A_{(k)})=f_{k}(X,A)$ and $\beta_{k-1}(X_{(k)},A_{(k)})=\beta_{k-1}(X,A)$. By Definition \ref{def::3.1}, the result follows immediately.
\end{proof}

	\begin{lemma}\label{lem::3.2}
		Let $(X,A)$ be a complex pair with $\dim X=d$. Then, for any $k\in\{0,\ldots,d\}$, the following statements hold:
\begin{enumerate}[(i)]
\item $\mathcal{F}_k(X,A)\neq \emptyset$;
\item $\mathcal{T}_k(X,A)\subseteq \mathcal{F}_k(X,A)$;
\item $\mathcal{T}_k(X,A)\neq \emptyset$ if and only if $\beta_{k-1}(X,A)=0$.
\end{enumerate}		
	\end{lemma}
	
	\begin{proof}
	
	We construct an element $\Gamma \in \mathcal{F}_k(X,A)$ as follows. First, set $\Gamma = X_{(k)}$. If $\beta_k(\Gamma,A_{(k)}) \neq 0$, then there exists a nonzero element $\gamma \in C_k(\Gamma,A_{(k)})$ in $\ker \partial_k(\Gamma,A_{(k)})$. Let $\sigma$ be a $k$-dimensional face of $\Gamma$ such that the coefficient of $[\sigma]$ in $\gamma$ is nonzero, and define $\Gamma' = \Gamma \setminus \{\sigma\}$. We then have $\beta_k(\Gamma',A_{(k)}) = \beta_k(\Gamma,A_{(k)}) - 1$ and $\beta_i(\Gamma',A_{(k)}) = \beta_i(\Gamma,A_{(k)})$ for all $i \leq k-2$. By Lemma \ref{lem::2.5}, we deduce that $\beta_{k-1}(\Gamma',A_{(k)}) = \beta_{k-1}(\Gamma,A_{(k)})$. Taking $\Gamma = \Gamma'$ and iterating this process successively, we eventually obtain a subcomplex $\Gamma$ satisfying $\beta_k(\Gamma,A_{(k)}) = 0$ and $\beta_{k-1}(\Gamma,A_{(k)}) = \beta_{k-1}(X,A)$. By Definition \ref{def::3.1} and Lemma \ref{lem::3.1}, we conclude that $\Gamma \in \mathcal{F}_k(X,A)$. Therefore, $\mathcal{F}_k(X,A)\neq \emptyset$, and this proves (i).
	
For (ii), let $\Upsilon \in \mathcal{T}_k(X,A)$ (if any exists). Since $A_{(k)} \subseteq \Upsilon \subseteq X_{(k)}$ and $\Upsilon_{(k-1)} = X_{(k-1)}$, it follows that $C_k(\Upsilon, A_{(k)}) \subseteq C_k(X, A)$, $C_{k-1}(\Upsilon, A_{(k)}) = C_{k-1}(X, A)$, and $C_{k-2}(\Upsilon, A_{(k)}) = C_{k-2}(X, A)$. Thus,  $\operatorname{im} \partial_{k}(\Upsilon,A_{(k)}) \subseteq \operatorname{im} \partial_{k}(X,A)$ and $\ker \partial_{k-1}(\Upsilon,A_{(k)}) = \ker \partial_{k-1}(X,A)$. This implies 
$\beta_{k-1}(X, A) \leq \beta_{k-1}(\Upsilon, A_{(k)}) = 0$,
and therefore $\beta_{k-1}(X, A) = \beta_{k-1}(\Upsilon, A_{(k)}) = 0$. Hence, by Definition \ref{def::3.1}, we conclude that $\Upsilon \in \mathcal{F}_k(X,A)$.

Finally, we consider (iii). The ``only if'' part follows immediately from the above proof. For the ``if'' part, assume that $\beta_{k-1}(X, A) = 0$. Take any $\Upsilon \in \mathcal{F}_k(X, A)$. Then we have $\beta_{k-1}(\Upsilon, A_{(k)}) = \beta_{k-1}(X, A) = 0$,
and consequently $\Upsilon \in \mathcal{T}_k(X,A)$ by Definition \ref{def::3.1}. Therefore, we conclude that $\mathcal{T}_k(X, A) = \mathcal{F}_k(X, A) \neq \emptyset$.
	\end{proof}
	
	In the remainder of this section, we always assume that $(X,A)$ is a complex pair with $\dim X=d$ satisfying $\beta_{k-1}(X,A)=0$ for some fixed $k\in\{0,\ldots,d\}$. In this case, Lemma \ref{lem::3.2} guarantees that $\mathcal{T}_k(X,A)\neq \emptyset$.

	\begin{lemma}\label{lem::3.3}
	We have
	\begin{equation*}
				f_k(X,A)=\beta_k(X_{(k)},A_{(k)})+\beta_{k-1}(X_{(k-1)},A_{(k-1)}).
		\end{equation*}
	\end{lemma}
	\begin{proof}
		By Lemma \ref{lem::2.5}, we have
		\begin{equation}\label{equ::11}
			\sum_{i=0}^k(-1)^if_i(X,A)=\sum_{i=0}^k(-1)^i\beta_i(X_{(k)},A_{(k)})
		\end{equation}
		and 
		\begin{equation}\label{equ::12}
			\sum_{i=0}^{k-1}(-1)^if_i(X,A)=\sum_{i=0}^{k-1}(-1)^i\beta_i(X_{(k-1)},A_{(k-1)}).
		\end{equation}
		Note that $\beta_i(X_{(j)},A_{(j)})=\beta_i(X,A)$ whenever $i\leq j-1$. Then it follows from \eqref{equ::11} and \eqref{equ::12} that
		\begin{equation*}
			\begin{aligned}
				f_k(X,A)&=\beta_k(X_{(k)},A_{(k)})-\beta_{k-1}(X,A)+\beta_{k-1}(X_{(k-1)},A_{(k-1)})\\
				&=\beta_k(X_{(k)},A_{(k)})+\beta_{k-1}(X_{(k-1)},A_{(k-1)}),
			\end{aligned}
		\end{equation*}
as desired.
	\end{proof}

For any subsets $B\subseteq X_k$ and $C\subseteq X_{k-1}$ with $|B|=|C|$, we denote by $\partial_k[B,C]$ the square submatrix of $\partial_k(X;\mathbb{R})$ with rows indexed by $C$ and columns indexed by $B$.	

Suppose that $\Upsilon \in \mathcal{T}_k(X,A)$ and $\Gamma \in \mathcal{F}_{k-1}(X,A)$. According to Definition \ref{def::3.1}, it is easy to see that $\Upsilon$ and $\Gamma$ can be respectively expressed as 
\[
\Upsilon = B \cup A_k \cup X_{(k-1)} \quad \text{and} \quad \Gamma = (X_{k-1} \setminus C) \cup X_{(k-2)},
\]
where $B$ and $C$ are chosen as follows:
\begin{enumerate}[(I)]
    \item $B$ is a subset of $X_k \setminus A_k$ with cardinality $|B| = f_k(X,A) - \beta_k(X_{(k)}, A_{(k)})$;
    \item  $C$ is a subset of $X_{k-1} \setminus A_{k-1}$ with cardinality $|C| = \beta_{k-1}(X_{(k-1)}, A_{(k-1)})$. 
\end{enumerate}
Moreover, by Lemma \ref{lem::3.3},  
\[|B|=f_k(X,A) - \beta_k(X_{(k)}, A_{(k)})=\beta_{k-1}(X_{(k-1)}, A_{(k-1)})=|C|.\]

Now suppose that $B \subseteq X_k \setminus A_k$ and $C \subseteq X_{k-1} \setminus A_{k-1}$ are two arbitrary subsets satisfying conditions (I) and (II). As above, we define  
\begin{equation*}  
    X_{B} = B \cup A_k \cup X_{(k-1)} \quad \text{and} \quad X_{C} = (X_{k-1} \setminus C) \cup X_{(k-2)}.  
\end{equation*}  
Observe that both $X_B$ and $X_{C}$ are subcomplexes of $X$, with dimensions $\dim X_{B} = k$ and $\dim X_{C} = k-1$, respectively. These subcomplexes satisfy the following inclusion relations:  
\[  
A_{(k-1)} \subseteq X_{C} \subseteq X_{(k-1)}, \quad X_{C} \subseteq X_{B}, \quad \text{and} \quad A_{(k)} \subseteq X_{B} \subseteq X_{(k)}.  
\]
By Lemma \ref{lem::2.3}, it is easy to see that the columns of $\partial_k(X_{B}, X_{C};\mathbb{R})$ indexed by $A_k$ are all zeros, and $\partial_k[B, C]$ is exactly the matrix obtained from $\partial_k(X_{B}, X_{C};\mathbb{R})$ by removing these zero columns. That is,
\begin{equation}\label{equ::13}
\partial_k(X_{B},X_{C};\mathbb{R})=
\left[\begin{matrix}
	\partial_k[B,C]& 0_{A_k}
\end{matrix}\right].
\end{equation}
 
	\begin{lemma}\label{lem::3.4}
		The matrix $\partial_k[B,C]$ is nonsingular if and only if $X_{B}\in\mathcal{T}_k(X,A)$ and $X_{C}\in\mathcal{F}_{k-1}(X,A)$.
	\end{lemma}
	\begin{proof}

	Consider the chain complex
\begin{equation}\label{equ::14}
    \begin{aligned}
        0 \longrightarrow {C_{k}(X_{B},X_{C})} \overset{\partial_{k}}{\longrightarrow}  {C_{k-1}(X_{B},X_{C})} \overset{\partial_{k-1}}{\longrightarrow}  {C_{k-2}(X_{B},X_{C})} \longrightarrow \cdots.
    \end{aligned}
\end{equation}
Observing that the free $\mathbb{Z}$-module $\mathcal{A}$ with basis $[A_k]$ can be identified with a submodule of $C_{k}(X_{B},X_{C})$ contained in $\ker \partial_{k}$, we define the quotient module
\[
C_{k}^{\square}(X_{B},X_{C}) := C_{k}(X_{B},X_{C})/\mathcal{A},
\]
and denote by $\partial^\square_{k}$ the induced quotient map from $C_{k}^{\square}(X_{B},X_{C})$ to $C_{k-1}(X_{B},X_{C})$. This yields the following chain complex induced by \eqref{equ::14}:
\begin{equation}\label{equ::15}
    \begin{aligned}
        0 \longrightarrow  C_{k}^{\square}(X_{B},X_{C}) \overset{\partial^\square_{k}}{\longrightarrow} C_{k-1}(X_{B},X_{C}) \overset{\partial_{k-1}}{\longrightarrow}  C_{k-2}(X_{B},X_{C}) \longrightarrow \cdots.
    \end{aligned}
\end{equation}
Let $H^{\square}_k(X_{B},X_{C})$ denote the group $\ker \partial^{\square}_{k}$. We have
\begin{equation}\label{equ::16}
H^{\square}_k(X_{B},X_{C})\oplus \mathbb{Z}^{|A_k|}\cong H_k(X_{B},X_{C}).
\end{equation}
Note that  $H_k(X_{B},X_{C};\mathbb{R})\cong \ker \partial_k(X_{B},X_{C};\mathbb{R})$ and $H_k(X_{B},X_{C})\cong \ker \partial_k(X_{B},X_{C};\mathbb{Z})$, and both of them are torsion-free. According to \eqref{equ::13}, the matrix $\partial_k[B,C]$ is nonsingular if and only if $H_k(X_{B},X_{C};\mathbb{R})\cong \mathbb{R}^{|A_k|}$. Recalling that $H_k(X_{B},X_{C};\mathbb{R})$ $\cong H_k(X_{B},X_{C})\otimes_{\mathbb{Z}}\mathbb{R}$ by \eqref{equ::2}, we conclude that $\partial_k[B,C]$ is nonsingular if and only if $H_k(X_{B},X_{C})\cong \mathbb{Z}^{|A_k|}$, which is the case if and only if  $H^{\square}_k(X_{B},X_{C})=0$ by \eqref{equ::16}. 

Recall that $A_{(k-1)} \subseteq X_{C} \subseteq X_{B}$ and $A_{(k)} \subseteq X_{B}$. Moreover, we see that $C_{k}^{\square}(X_{B},X_{C}) \cong C_{k}(X_{B},A_{(k)})/C_k(X_{C},A_{(k-1)})$. Thus we obtain the following commutative diagram associated with \eqref{equ::15}:
		\begin{equation}\label{equ::17}
			\begin{tikzcd}
				& 0 \arrow[d]                      & 0 \arrow[d]                        & 0 \arrow[d]                        &        \\
				0 \arrow[r] & {C_k(X_{C},A_{(k-1)})}=0 \arrow[r,"\partial_k"] \arrow[d,"i"]   & {C_{k-1}(X_{C},A_{(k-1)})} \arrow[r,"\partial_{k-1}"] \arrow[d,"i"] & {C_{k-2}(X_{C},A_{(k-1)})} \arrow[d,"i"] \arrow[r] & \cdots \\
				0 \arrow[r] & {C_{k}(X_{B},A_{(k)})} \arrow[r,"\partial_{k}"] \arrow[d,"j"] & {C_{k-1}(X_{B},A_{(k)})} \arrow[r,"\partial_{k-1}"] \arrow[d,"j"] & {C_{k-2}(X_{B},A_{(k)})} \arrow[d,"j"] \arrow[r] & \cdots \\
				0 \arrow[r] & {C^{\square}_{k}(X_{B},X_{C})} \arrow[r,"\partial^{\square}_{k}"] \arrow[d] & {C_{k-1}(X_{B},X_{C})} \arrow[r,"\partial_{k-1}"] \arrow[d] & {C_{k-2}(X_{B},X_{C})} =0\arrow[d] \arrow[r] & \cdots \\
				& 0                                & 0                                  & 0                                  &       
			\end{tikzcd},
		\end{equation}
		where $i$ is the inclusion, $j$ is the quotient map, and $\partial_{\bullet}$ are the corresponding relative boundary maps. Since the columns are exact sequences and the rows are chain complexes, the diagram  \eqref{equ::17}  gives a  short exact sequence of chain complexes. By the zig-zag lemma (see, for example, \cite[p. 116]{Hat02}), the diagram \eqref{equ::17} stretches out into a long exact sequence of relative homology groups:
		\begin{equation}\label{equ::18}
			\begin{aligned}
				 &H_k(X_{C},A_{(k-1)})=0\longrightarrow H_k(X_{B},A_{(k)})\longrightarrow H^{\square}_k(X_{B},X_{C}) \\
				 \longrightarrow&H_{k-1}(X_{C},A_{(k-1)})\longrightarrow H_{k-1}(X_{B},A_{(k)})\longrightarrow 
				 H_{k-1}(X_{B},X_{C})\\
				\longrightarrow  &H_{k-2}(X_{C},A_{(k-1)})\longrightarrow H_{k-2}(X_{B},A_{(k)})\longrightarrow H_{k-2}(X_{B},X_{C})=0.
			\end{aligned}
		\end{equation}
		
		If $X_{B} \in \mathcal{T}_k(X,A)$ and $X_{C} \in \mathcal{F}_{k-1}(X,A)$, by Definition \ref{def::3.1}, we have $\beta_k(X_{B},A_{(k)}) = 0$ and $\beta_{k-1}(X_{C},A_{(k-1)}) = 0$. This implies that $H_k(X_{B},A_{(k)}) = 0$ and $H_{k-1}(X_{C},A_{(k-1)}) = 0$, since both groups are torsion-free. Therefore, from the long exact sequence \eqref{equ::18}, we immediately deduce that $H^{\square}_k(X_{B},X_{C}) = 0$, or equivalently, $\partial_k[B,C]$ is nonsingular.
		
		Conversely, if $\partial_k[B,C]$ is nonsingular, then $H^{\square}_k(X_{B},X_{C}) = 0$. From the long exact sequence \eqref{equ::18}, we obtain $H_k(X_{B},A_{(k)}) = 0$, which implies $\beta_k(X_{B},A_{(k)}) = 0$. Since $f_k(X_{B},A_{(k)}) = |B| = f_k(X,A) - \beta_k(X_{(k)},A_{(k)})$ and $\beta_{k-1}(X,A) = 0$, by Lemma \ref{lem::3.1} and Definition \ref{def::3.1}, we obtain $\beta_{k-1}(X_{B},A_{(k)}) = \beta_{k-1}(X,A) = 0$. Thus, $X_{B} \in \mathcal{T}_k(X,A)$ and $H_{k-1}(X_{B},A_{(k)})$ is a finite group. From the long exact sequence \eqref{equ::18} and the fact that $H_{k-1}(X_{C},A_{(k-1)})$ is torsion-free, we deduce that $H_{k-1}(X_{C},A_{(k-1)}) = 0$, and consequently $\beta_{k-1}(X_{C},A_{(k-1)}) = 0$. Moreover, $f_{k-1}(X_{C},A_{(k-1)}) = |X_{k-1}| - |A_{k-1}| - |C| = f_{k-1}(X,A) - \beta_{k-1}(X_{(k-1)},A_{(k-1)})$. Therefore, again by Lemma \ref{lem::3.1} and Definition \ref{def::3.1}, we conclude that $X_{C} \in \mathcal{F}_{k-1}(X,A)$.
	\end{proof}
	
	\begin{lemma}\label{lem::3.5}
		If  $\partial_k[B,C]$ is nonsingular, then 
		\begin{equation*}
			|\det \partial_k[B,C]|=|H_{k-1}(X_{B},A_{(k)})|\cdot \frac{|H_{k-2}(X_{C},A_{(k-1)})/\mathbb{Z}^{\beta}|}{|H_{k-2}(X,A)/\mathbb{Z}^{\beta}|},
		\end{equation*}
		where $\beta=\beta_{k-2}(X,A)$.
	\end{lemma}
	\begin{proof}

	Since $f_{k-2}(X_{B}, X_{C}) = 0$, $f_{k-1}(X_{B}, X_{C}) = |B| = |C|$, and $f_k(X_{B}, X_{C}) = |B| + |A_k|$, we have  
$C_{k-2}(X_{B}, X_{C}) = 0$, $C_{k-1}(X_{B}, X_{C}) \cong \mathbb{Z}^{|B|}$, and $C_k(X_{B}, X_{C}) \cong \mathbb{Z}^{|B| + |A_k|}$. Then, $\ker \partial_{k-1}(X_{B}, X_{C}) = C_{k-1}(X_{B}, X_{C}) \cong \mathbb{Z}^{|B|}$ and $\operatorname{im} \partial_k(X_{B}, X_{C}) = \partial_k(C_k(X_{B}, X_{C}))\subseteq C_{k-1}(X_{B}, X_{C})$. Let $M$ be the  $\mathbb{Z}$-matrix representation of $\partial_k(X_{B}, X_{C})$ under the basis corresponding to  $[B] \cup [A_k]$. We have
\[
H_{k-1}(X_{B}, X_{C}) = \ker \partial_{k-1}(X_{B}, X_{C}) / \operatorname{im} \partial_k(X_{B}, X_{C}) \cong \mathbb{Z}^{|B|} / M(\mathbb{Z}^{|B| + |A_k|}).
\]
Note that $\partial_k(X_{B}, X_{C}; \mathbb{R})$ is also a $\mathbb{Z}$-matrix. By \eqref{equ::3} and \eqref{equ::13}, we assert that
\[
M = \partial_k(X_{B}, X_{C}; \mathbb{R}) = \begin{bmatrix}
\partial_k[B,C] & 0_{A_k}
\end{bmatrix}.
\]
 Hence, $M(\mathbb{Z}^{|B| + |A_k|}) \subseteq \mathbb{Z}^{|C|} = \mathbb{Z}^{|B|}$. Since $\partial_k[B,C]$ is nonsingular, the rank of $M$ is equal to $|B|$. Using the Smith normal form, we can find a generating set $\{w_1, \ldots, w_{|B|}\}$ of $\mathbb{Z}^{|B|}$ and a generating set $\{v_1, \ldots, v_{|B|}\}$ of $M(\mathbb{Z}^{|B| + |A_k|})$ such that
\[
v_i = \alpha_i w_i, \quad \text{for } 1 \leq i \leq |B|,
\]
where the $\alpha_i$ are invariant factors satisfying $\alpha_i \mid \alpha_{i+1}$. Therefore,
\[
\mathbb{Z}^{|B|} / M(\mathbb{Z}^{|B| + |A_k|}) \cong \mathbb{Z}/\alpha_1\mathbb{Z} \oplus \cdots \oplus \mathbb{Z}/\alpha_{|B|}\mathbb{Z},
\]
which implies that 
\begin{equation}\label{equ::19}
|H_{k-1}(X_{B}, X_{C})| =\prod_{i=1}^{|B|}\alpha_i= |\det \partial_k[B,C]|.
\end{equation}

Recall that $\partial_k[B,C]$ is nonsingular. By Lemma \ref{lem::3.4}, we have $X_{C} \in \mathcal{F}_{k-1}(X,A)$, and hence $H_{k-1}(X_{C},A_{(k-1)}) = 0$. Then the long exact sequence \eqref{equ::18} yields the following exact sequence:
\begin{equation*}
    \begin{aligned}
        0 &\longrightarrow  H_{k-1}(X_{B},A_{(k)}) \longrightarrow  H_{k-1}(X_{B},X_{C}) 
       \longrightarrow  H_{k-2}(X_{C},A_{(k-1)}) \\
        & \overset{f}{\longrightarrow} H_{k-2}(X_{B},A_{(k)})=H_{k-2}(X,A) \longrightarrow 0.
    \end{aligned}
\end{equation*}
It follows that both
\begin{equation}\label{equ::20}
 \begin{aligned}
        0 \longrightarrow H_{k-1}(X_{B},A_{(k)})  \longrightarrow H_{k-1}(X_{B},X_{C})  \longrightarrow \ker f  \longrightarrow 0
    \end{aligned}
\end{equation}
and 
\begin{equation}\label{equ::21}
\begin{aligned}
        0 \longrightarrow \ker f \longrightarrow H_{k-2}(X_{C},A_{(k-1)}) \overset{f}{\longrightarrow} H_{k-2}(X,A) \longrightarrow 0
    \end{aligned}
\end{equation}
are short exact sequences. Let $r(\ker f)$ denote the rank of the largest free $\mathbb{Z}$-module summand of $\ker f$. From \eqref{equ::20} and \eqref{equ::21}, we obtain 
\begin{equation}\label{equ::22}
    \beta_{k-1}(X_{B},A_{(k)}) + r(\ker f) = \beta_{k-1}(X_{B},X_{C}),
\end{equation} 
and
\begin{equation}\label{equ::23}
    r(\ker f) + \beta_{k-2}(X,A) = \beta_{k-2}(X_{C},A_{(k-1)}),
\end{equation}
respectively. Since $X_{C} \in \mathcal{F}_{k-1}(X,A)$, we have $\beta_{k-2}(X_{C},A_{(k-1)}) = \beta_{k-2}(X,A)$, and hence $r(\ker f) = 0$ by \eqref{equ::23}. By Lemma \ref{lem::3.4}, we obtain $X_{B}\in\mathcal{T}_k(X,A)$, which implies $\beta_{k-1}(X_{B},A_{(k)}) = 0$. Combining this with \eqref{equ::22} and the fact that $r(\ker f) = 0$, we get $\beta_{k-1}(X_{B},X_{C})= 0$. Therefore, all groups in the short exact sequence \eqref{equ::20} are finite. By the first isomorphism theorem,
\begin{equation}\label{equ::24}
    |\ker f|=\frac{|H_{k-1}(X_{B},X_{C})|}{|H_{k-1}(X_{B},A_{(k)})|}.
\end{equation}

Now we compute $|\ker f|$ in an alternative way. Let $\beta = \beta_{k-2}(X_{C}, A_{(k-1)}) = \beta_{k-2}(X, A)$. We can express $H_{k-2}(X_{C}, A_{(k-1)})$ as 
\[
H_{k-2}(X_{C}, A_{(k-1)})=(H_{k-2}(X_{C}, A_{(k-1)}) / \mathbb{Z}^{\beta}) \oplus \mathbb{Z}^{\beta}.
\]
Since $r(\ker f) = 0$, we can regard $\ker f$ as a subgroup of $H_{k-2}(X_{C}, A_{(k-1)}) / \mathbb{Z}^{\beta}$. Applying the first isomorphism theorem to the short exact sequence \eqref{equ::21}, we get
\[
\big( (H_{k-2}(X_{C}, A_{(k-1)}) / \mathbb{Z}^{\beta}) / \ker f \big) \oplus \mathbb{Z}^{\beta} \cong (H_{k-2}(X, A) / \mathbb{Z}^{\beta}) \oplus \mathbb{Z}^{\beta},
\]
which implies
\begin{equation}\label{equ::25}
|\ker f| = \frac{|H_{k-2}(X_{C}, A_{(k-1)}) / \mathbb{Z}^{\beta}|}{|H_{k-2}(X, A) / \mathbb{Z}^{\beta}|}.
\end{equation}
Combining \eqref{equ::19} with \eqref{equ::24} and \eqref{equ::25}, we obtain the desired result.
	\end{proof}
	
	Now we are ready to give the proof of Theorem \ref{thm::main1}.
	\renewcommand\proofname{\bf{Proof of Theorem \ref{thm::main1}}} 
	\begin{proof}
		By Lemma \refeq{lem::3.2}, it suffices to prove (i) and (ii) under the assumption that $\beta_{k-1}(X,A)=0$.
		
First we consider (i). Note that  $L^{\operatorname{ud}}_{k-1}(X,A)=L^{\operatorname{ud}}_{k-1}(X_{(k)},A_{(k)})$ is  a square matrix of order $f_{k-1}(X,A)=f_{k-1}(X_{(k)},A_{(k)})$. By Lemma \ref{lem::2.4}, the  rank of $L^{\operatorname{ud}}_{k-1}(X,A)$ is equal to
\[
\chi_{k-1}(X_{(k)},A_{(k)})=f_{k}(X_{(k)},A_{(k)})-\beta_{k}(X_{(k)},A_{(k)})=f_{k}(X,A)-\beta_{k}(X_{(k)},A_{(k)}).
\] 
Let $p(y):=\det(yI-L^{\operatorname{ud}}_{k-1}(X,A))$ denote the characteristic polynomial of $L^{\operatorname{ud}}_{k-1}(X,A)$, and let $\pi$ be the product of all nonzero eigenvalues of $L^{\operatorname{ud}}_{k-1}(X,A)$. Then $\pi$ is given (up to sign) by the coefficient of the monomial $y^{f_{k-1}(X,A)-(f_k(X,A)-\beta_{k}(X_{(k)},A_{(k)}))}$ in $p(y)$. That is, 
		\begin{equation*}
			\pi=\sum_{\substack{C\in X_{k-1}\setminus A_{k-1}\\|C|=f_{k}(X,A)-\beta_{k}(X_{(k)},A_{(k)})}}\det L_C,
		\end{equation*}
		where $L_C$ is the principle submatrix of $L^{\operatorname{ud}}_{d-1}(X,A)$ indexed by $C$. By the Cauchy-Binet formula (cf. \cite[p. 28]{Hor13}), we have 
		\begin{equation*}
			\det L_C=\sum_{\substack{B\subseteq X_k\setminus A_k\\ |B|=|C|}} (\det\partial_d[B,C])^2.
		\end{equation*}
Recall that, by Lemma \ref{lem::3.4}, $\det\partial_k[B,C]\neq 0$ if and only if $X_B\in \mathcal{T}_k(X,A)$ and $X_{C}\in\mathcal{F}_{k-1}(X,A)$. Thus,
		\begin{equation*}
			\pi=\sum_{\substack{B\subseteq X_k\setminus A_k\\ X_B\in\mathcal{T}_k(X,A)}}\sum_{\substack{C\subseteq X_{k-1}\setminus A_{k-1}\\ X_{C}\in\mathcal{F}_{k-1}(X,A)}}(\det\partial_k[B,C])^2.
		\end{equation*}
Combining this with Lemma \ref{lem::3.5} and the arguments below Lemma \ref{lem::3.3}, we deduce that
\begin{equation*}
\begin{aligned}
				\pi&=\sum_{\substack{B\subseteq X_k\setminus A_k\\ X_B\in\mathcal{T}_k(X,A)}} |H_{k-1}(X_B,A_{(k)})|^2  \cdot \sum_{\substack{C\subseteq X_{k-1}\setminus A_{k-1}\\ X_{C}\in\mathcal{F}_{k-1}(X,A)}}\frac{|H_{k-2}(X_C,A_{(k-1)})/\mathbb{Z}^{\beta}|^2}{|H_{k-2}(X,A)/\mathbb{Z}^{\beta}|^2}\\
&=\sum_{\Upsilon\in \mathcal{T}_k(X,A)}|H_{k-1}(\Upsilon,A_{(k)})|^2 \cdot \sum_{\Gamma\in \mathcal{F}_{k-1}(X,A)}\frac{|H_{k-2}(\Gamma,A_{(k-1)})/\mathbb{Z}^{\beta}|^2}{|H_{k-2}(X,A)/\mathbb{Z}^{\beta}|^2}.
				\end{aligned}
			\end{equation*}
This proves (i).

		Now we consider (ii). Let $C=X_{k-1}\setminus \Gamma_{k-1}$. We see that $C\subseteq X_{k-1}\setminus A_{k-1}$, $|C|=\beta_{k-1}(X_{(k-1)}, A_{(k-1)})=f_k(X,A) - \beta_k(X_{(k)}, A_{(k)})$ and $\Gamma=X_C$. 
		By Lemma \ref{lem::2.3} and the Cauchy-Binet formula, we get 
		\begin{equation*}
			\det L_{k-1}^{\operatorname{ud}}(X,\Gamma)=\sum_{\substack{B\subseteq X_k\\|B|=|C|}}(\det\partial_k[B,C])^2.
		\end{equation*}
Since $A_k \cap C = \emptyset$, we have $\det \partial_k[B,C] = 0$ for all $B \subseteq X_k$ satisfying $B \cap A_k \neq \emptyset$ and $|B| = |C|$. Moreover, because $X_C = \Gamma \in \mathcal{F}_{k-1}(X,A)$, it follows from Lemma \ref{lem::3.4} that for any $B \subseteq X_k \setminus A_k$ with $|B| = |C|$, $\det \partial_k[B,C] \neq 0$ if and only if $X_B \in \mathcal{T}_k(X,A)$. Thus,
		\begin{equation*}
			\det L_{k-1}^{\operatorname{ud}}(X,\Gamma)=\sum_{\substack{B \subseteq X_k \setminus A_k\\X_{B}\in\mathcal{T}_k(X,A)}}(\det\partial_k[B,C])^2.
		\end{equation*}
From Lemma \ref{lem::3.5} and the arguments below Lemma \ref{lem::3.3}, we obtain
		\begin{equation*}
			\begin{aligned}
				\det L_{k-1}^{\operatorname{ud}}(X,\Gamma)&=\sum_{\substack{B \subseteq X_k \setminus A_k\\X_{B}\in\mathcal{T}_k(X,A)}}|H_{k-1}(X_{B},A_{(k)})|^2\cdot \frac{|H_{k-2}(X_{C},A_{(k-1)})/\mathbb{Z}^{\beta}|^2}{|H_{k-2}(X,A)/\mathbb{Z}^{\beta}|^2}\\
				&=\frac{|H_{k-2}(\Gamma,A_{(k-1)})/\mathbb{Z}^{\beta}|^2}{|H_{k-2}(X,A)/\mathbb{Z}^{\beta}|^2}\sum_{\Upsilon\in \mathcal{T}_k(X,A)}|H_{k-1}(\Upsilon,A_{(k)})|^2,
			\end{aligned}
		\end{equation*}
and  (ii) follows.
	\end{proof}
	
\section{Proofs of Theorems \ref{thm::main2}--\ref{thm::main4} and \ref{thm::main5}}\label{section::4}

Let $X$ be a simplicial complex. For any $\sigma\in X$,  the \textit{link} of $\sigma$ in $X$ is defined as
	$$\mathrm{lk}(X,\sigma)=\{\tau\in X:\tau\cup\sigma\in X,\tau\cap\sigma=\emptyset\}.$$ 	
	\begin{lemma}[{\cite[Lemma 1.4]{Lew20}, \cite[Lemma 4.3]{ZHL26}}]\label{lem::4.1}
		Let $X$ be a simplicial complex on vertex set $V$, with $h(X)=h$. Let $\sigma\in X_k$ and $v\in V\setminus \sigma$. If $v\notin \mathrm{lk}(X,\sigma)$, then
		\begin{equation*}			
		|\{\tau\in\sigma_{k-1}:v\in\mathrm{lk}(X,\tau)\}| \leq h.
		\end{equation*}
		\end{lemma}
Based on Lemma \ref{lem::4.1} and Theorem \ref{thm::Lew}, we now present the proof of Theorem  \ref{thm::main2}.

\renewcommand\proofname{\bf{Proof of Theorem \ref{thm::main2}}}
\begin{proof}
Since $\deg_{X}(\sigma) \leq 1$ for all $\sigma \in A_{k-1}$, we have the following two facts:
\begin{itemize}
    \item $\sigma \cap \tau \notin A_{k-1}$ for any $\sigma, \tau \in X_k$;
    \item $X'_k=X_k\setminus A_k$ and $X'_{k+1}=X_{k+1}$.
\end{itemize}
Combining these facts with Lemma \ref{lem::2.2}, we see that for any $\sigma,\tau\in X_k\setminus A_k=X_k'$,
		\begin{equation}\label{equ::27}
		\begin{aligned}
			L_k(X,A)(\sigma,\tau)&=\begin{cases} 
				\deg_X(\sigma)+|\sigma_{k-1}\setminus A_{k-1}| & \mbox{if $\sigma=\tau$}, \\
				(-1)^{\epsilon(\sigma,\tau)} &\mbox{if $\sigma\cup\tau\notin X_{k+1},~\sigma\cap\tau\in X_{k-1}\setminus A_{k-1}$}, \\
				0 & \mbox{otherwise},
			\end{cases}\\
			&=\begin{cases} 
				\deg_{X'}(\sigma)+|\sigma_{k-1}\setminus A_{k-1}| & \mbox{if $\sigma=\tau$}, \\
				(-1)^{\epsilon(\sigma,\tau)} &\mbox{if $\sigma\cup\tau\notin X'_{k+1},~\sigma\cap\tau\in X'_{k-1}$}, \\
				0 & \mbox{otherwise}.
				\end{cases}
				\end{aligned}
		\end{equation}
Let $V'$ denote the vertex set of $X'$. From the Ger\v{s}gorin circle theorem (cf. \cite[Theorem 6.1.1]{Hor13}), \eqref{equ::27} and  Lemma \ref{lem::4.1}, we obtain 
		\begin{align*}
			\mu_k(X,A)&\geq \min_{\sigma\in X'_k} \Bigg( L_k(X,A)(\sigma,\sigma)-\sum_{\substack{\tau\in X'_k\\\tau\neq \sigma}}|L_k(X,A)(\sigma,\tau)|\Bigg)\\
			&= \min_{\sigma\in X'_k} \left(\deg_{X'}(\sigma)+|\sigma_{k-1}\setminus A_{k-1}|-|\{\tau\in X'_k:|\sigma\cap \tau|=k,\sigma\cup\tau\notin X'_{k+1}\}|\right)
			\\
			&= \min_{\sigma\in X'_k} \Bigg(\deg_{X'}(\sigma)+|\sigma_{k-1}\setminus A_{k-1}|-\sum_{\eta\in \sigma_{k-1}}|\{v\in V'\setminus \sigma:v\in\operatorname{lk}(X',\eta),v\notin \operatorname{lk}(X',\sigma)\}|\Bigg)\\
			&=\min_{\sigma\in X'_k} \Bigg(\deg_{X'}(\sigma)+|\sigma_{k-1}\setminus A_{k-1}|-\sum_{\substack{v\in V'\setminus \sigma\\ v\notin \operatorname{lk}(X',\sigma)}}|\{\eta\in\sigma_{k-1}:v\in\operatorname{lk}(X',\eta)\}|\Bigg)\\
			&\geq \min_{\sigma\in X'_k} \Bigg(\deg_{X'}(\sigma)+|\sigma_{k-1}\setminus A_{k-1}|-\sum_{\substack{v\in V'\setminus \sigma\\ v\notin \operatorname{lk}(X',\sigma)}}h'\Bigg)\\
			&\geq \min_{\sigma\in X'_k} \left(\deg_{X'}(\sigma)+k+1-|\sigma_{k-1}\cap A_{k-1}|-(n'-k-1-\deg_{X'}(\sigma))h'\right)\\
			&=\min_{\sigma\in X'_k}\left((h'+1)\deg_{X'}(\sigma)-|\sigma_{k-1}\cap A_{k-1}|\right)+(h'+1)(k+1)-h'n'\\
			&\geq (h' + 1)(k + 1) - h'n' - \max_{\sigma \in X'_k} |\sigma_{k-1} \cap A_{k-1}|.
		\end{align*}
This proves the first part of the theorem.

Now we consider the case where equality holds.   By Lemma \ref{lem::2.1}, for any $\sigma,\tau\in X'_k$,
		\begin{equation}\label{equ::28}
				L_k(X')(\sigma,\tau)=\begin{cases} 
					\deg_{X'}(\sigma)+k+1 & \mbox{if $\sigma=\tau$}, \\
					(-1)^{\epsilon(\sigma,\tau)} &\mbox{if $\sigma\cup\tau\notin X'_{k+1},~\sigma\cap\tau\in X'_{k-1}$}, \\
					0 & \mbox{otherwise}.
				\end{cases}
		\end{equation}
Let $R$ be the diagonal matrix indexed by $X'_k$ with entries $R(\sigma,\sigma) = |\sigma_{k-1} \cap A_{k-1}|$ for $\sigma \in X'_k$. From \eqref{equ::27} and \eqref{equ::28}, we obtain the relation $L_k(X,A) = L_k(X') - R$. Then, by the Weyl inequality (cf. \cite[Theorem 4.3.1]{Hor13}),
\begin{equation}\label{equ::29}
    \mu_k(X,A) \geq \mu_k(X') - \max_{\sigma \in X'_k} |\sigma_{k-1} \cap A_{k-1}|.
\end{equation}
Moreover, equality $\mu_k(X,A) = \mu_k(X') - \max_{\sigma \in X'_k} |\sigma_{k-1} \cap A_{k-1}|$ holds if and only if there exists a nonzero vector $\boldsymbol{x}$ satisfying $L_k(X')\boldsymbol{x}=\mu_k(X')\boldsymbol{x}$, $R\boldsymbol{x}=(\max_{\sigma\in X'_k}|\sigma_{k-1}\cap A_{k-1}|)\boldsymbol{x}$, and $L_k(X,A)\boldsymbol{x}=\mu_k(X,A)\boldsymbol{x}$.
On the other hand, Theorem \ref{thm::Lew} implies that
\begin{equation}\label{equ::30}
    \mu_k(X') \geq (h' + 1)(k + 1) - h'n'.
\end{equation}

If $\mu_k(X,A)=(h'+1)(k+1)-h'n'-\max_{\sigma\in X'_k}|\sigma_{k-1}\cap A_{k-1}|$, then from \eqref{equ::29} and \eqref{equ::30}, we immediately obtain 
\[
\mu_k(X')=(h'+1)(k+1)-h'n'\]
 and 
 \[
 \mu_k(X,A)=\mu_k(X')-\max_{\sigma\in X'_k}|\sigma_{k-1}\cap A_{k-1}|.
 \] 
 By Theorem \ref{thm::Lew}, the first equality implies that
 \[
 X'\cong(\Delta^{h'}_{(h'-1)})^{*(n'-k-1)}*\Delta^{(h'+1)(k+1)-h'n'-1}.
 \] 
 Moreover, by the above arguments, the second  equality suggests that there exists a nonzero vector $\boldsymbol{x}$ such that $L_k(X')\boldsymbol{x}=\mu_k(X')\boldsymbol{x}$ and $R\boldsymbol{x}=(\max_{\sigma\in X'_k}|\sigma_{k-1}\cap A_{k-1}|)\boldsymbol{x}$. Since $\boldsymbol{x}$ has no zero entries by Theorem \ref{thm::Lew}, we assert that $R$ must be a scalar matrix, i.e., $|\sigma_{k-1}\cap A_{k-1}|=c$  (a constant) for all $\sigma\in X'_k$. 

Conversely, if $X'\cong(\Delta^{h'}_{(h'-1)})^{*(n'-k-1)}*\Delta^{(h'+1)(k+1)-h'n'-1}$ and $|\sigma_{k-1}\cap A_{k-1}|=c$ (a constant) for all $\sigma\in X'_k$, then $L_k(X,A)=L_k(X')-R=L_k(X')-cI$, where $I$ is the identity matrix. Combining this with Theorem \ref{thm::Lew}, we obtain
\begin{equation*}
\mu_k(X,A)=\mu_k(X')-c=(h'+1)(k+1)-h'n'-\max_{\sigma\in X'_k}|\sigma_{k-1}\cap A_{k-1}|.
\end{equation*}
The result follows.
\end{proof}
		
To prove Theorem \ref{thm::main3}, we will require the notation of $k$-th additive compound matrices. Let $M$ be an $n \times n$ matrix over a field $\mathbb{F}$, and let $1 \leq k \leq n$. The \textit{$k$-th additive compound} of $M$ is an $\binom{n}{k} \times \binom{n}{k}$ matrix $M^{[k]}$, whose rows and columns are indexed by the $k$-subsets of $[n]$, defined by
\begin{equation}\label{equ::31}
    M^{[k]}(\sigma, \tau) = 
    \begin{cases} 
        \sum_{i \in \sigma} M(i, i) & \text{if $\sigma = \tau$}, \\
        (-1)^{\epsilon(\sigma, \tau)} M(i, j) & \text{if $|\sigma \cap \tau| = k - 1$, $\sigma \setminus \tau = \{i\}$, and $\tau \setminus \sigma = \{j\}$}, \\
        0 & \text{otherwise}.
    \end{cases}
\end{equation}

	\begin{lemma}[{\cite[Theorem 2.1]{Fie74}}]\label{lem::4.2}
	    Let $M$ be an $n\times n$ matrix over a field $\mathbb{F}$, with eigenvalues $\lambda_1,\ldots,\lambda_n$. Then, the $k$-th additive compound $M^{[k]}$ has eigenvalues $\lambda_{i_1}+\cdots+\lambda_{i_k}$, for $1\leq i_1<\cdots<i_k\leq n$.
	\end{lemma}

	\begin{lemma}[{\cite[Lemma 3.1]{Lew23}}]\label{lem::4.3}
	Let $G = (V, E)$ be a graph, and let $X = X(G)$. Let $w : V \to \mathbb{R}_{\geq 0}$. Then, for all $k \geq 0$ and $\sigma \in X_k$,
\[
\Bigg(\sum_{v \in \sigma} \sum_{\substack{u\in V\\ \{u,v\}\in E}} w(u) \Bigg) - \sum_{v \in \operatorname{lk}(X,\sigma)} w(v) \leq k \sum_{v \in V} w(v).
\]	
	\end{lemma}

\renewcommand\proofname{\bf{Proof of Theorem \ref{thm::main3}}}
\begin{proof}
	Let $V'$ and $E'$ denote the vertex set and edge set of $G_{X'}$. By Lemma \ref{lem::2.1}, we have 
\begin{equation}\label{equ::32}
    L_0(X')(u,v) = \begin{cases} 
        \deg_{X'}(u) + 1 & \text{if } u = v, \\
        1 & \text{if } \{u,v\} \notin E', \\
        0 & \text{otherwise},
    \end{cases}
\end{equation}
where $u,v \in V'$. Note that $\mu_0(X') = \lambda_2(G_{X'})$.

		 Let $P$ be the principal submatrix of $L_0(X')^{[k+1]}$ with rows and columns indexed by $X'_k$. From the Cauchy interlacing theorem (cf. \cite[Theorem 4.3.17]{Hor13}) and Lemma \ref{lem::4.2}, we obtain
		\begin{equation}\label{equ::33}
			\lambda_{\min}(P)\geq \lambda_{\min}(L_0(X')^{[k+1]})\geq (k+1)\mu_0(X')=(k+1)\lambda_2(G_{X'}).
		\end{equation}

		For $\sigma, \tau \in X'_k$ with $|\sigma \cap \tau| = k$, denote $(\sigma \setminus \tau) \cup (\tau \setminus \sigma) = \{u, v\}$.  If $\{u, v\} \in E'$, then $\sigma \cup \tau \in X_{k+1} = X'_{k+1}$, since both $\sigma$ and $\tau$ span complete subgraphs of $G_X$ of order $k+1$. Conversely, if $\sigma \cup \tau \in X'_{k+1}$, then $\{u, v\} \in E'$. Therefore, $\{u, v\} \notin E'$ if and only if $\sigma \cup \tau \notin X'_{k+1}$.  By  \eqref{equ::31} and \eqref{equ::32},  for any $\sigma,\tau\in X'_k$, 
		\begin{equation*}\label{equ::matrix_P}
			P(\sigma,\tau)=\begin{cases} 
				k+1+\sum_{u\in\sigma}\deg_{X'}(u) & \mbox{if $\sigma=\tau$}, \\
				(-1)^{\epsilon(\sigma,\tau)} &\mbox{if $|\sigma\cap\tau|=k,~\sigma\cup\tau\notin X'_{k+1} $}, \\
				0 & \mbox{otherwise}.
			\end{cases}
		\end{equation*}
Combining this with \eqref{equ::27}, we obtain 
		\begin{equation*}
			L_k(X,A)=P-R,
		\end{equation*}
where $R$ is the diagonal matrix defined by 
\[
R(\sigma,\sigma) = \sum_{u\in\sigma}\deg_{X'}(u)-\deg_{X'}(\sigma)+|\sigma_{k-1}\cap A_{k-1}|,\quad\text{for } \sigma \in X'_k.
\]  
Therefore, it follows from  the Weyl inequality, \eqref{equ::33} and Lemma \ref{lem::4.3} that
		\begin{equation*}
			\begin{aligned}
				\mu_k(X,A)&\geq \lambda_{\min}(P)-\max_{\sigma\in X'_k}\left(\sum_{u\in\sigma}\deg_{X'}(u)-\deg_{X'}(\sigma)+|\sigma_{k-1}\cap A_{k-1}|\right)\\
				&\geq (k+1)\lambda_2(G_{X'})-\max_{\sigma\in X'_k}\left(\sum_{u\in\sigma}\deg_{X'}(u)-\deg_{X'}(\sigma)\right)-\max_{\sigma\in X'_k}|\sigma_{k-1}\cap A_{k-1}|\\
				&\geq (k+1)\lambda_2(G_{X'})-kn'-\max_{\sigma\in X'_k}|\sigma_{k-1}\cap A_{k-1}|.
			\end{aligned}
		\end{equation*}
Moreover, if 
$$\lambda_2(G_{X'})>\frac{1}{k+1}\left(kn'+\max_{\sigma\in X'_k}|\sigma_{k-1}\cap A_{k-1}|\right),$$ 
then $\mu_k(X,A)>0$, and hence by Theorem \ref{thm::1.2},  we have $H_k(X,A;\mathbb{R})=0$.
		
		This completes the proof. 
	\end{proof}

Next, we present the proof of Theorem \ref{thm::main4}.

\renewcommand\proofname{\bf{Proof of Theorem \ref{thm::main4}}}	
	\begin{proof}
	Let $\boldsymbol{x}$ be an eigenvector of $L_k(X,A)$ with respect to $\mu_k(X,A)$, and let $\boldsymbol{y}$ be the vector defined on $X_k$ given by
\[
\boldsymbol{y}(\sigma) = \begin{cases}
\boldsymbol{x}(\sigma) & \text{if } \sigma \in X_k \setminus A_k, \\
0 & \text{if } \sigma \in A_k.
\end{cases}
\]	
By Lemmas \ref{lem::2.1} and \ref{lem::2.2}, 
		\begin{equation*}
			\begin{aligned}
				\boldsymbol{y}^{\top}L_k(X)\boldsymbol{y}-\boldsymbol{x}^{\top}L_k(X,A)\boldsymbol{x}
				&=\sum_{\sigma\in X_k\setminus A_k}|\sigma_{k-1}\cap A_{k-1}|\boldsymbol{x}^2(\sigma)+\sum_{\substack{\sigma,\tau\in X_k\setminus A_k\\\sigma\cap\tau\in A_{k-1}}}(-1)^{\epsilon(\sigma,\tau)}\boldsymbol{x}(\sigma)\boldsymbol{x}(\tau)\\
				&=\boldsymbol{x}^{\top}M\boldsymbol{x},
			\end{aligned}
		\end{equation*}
		where $M$ is the matrix with rows and columns  indexed by $X_k\setminus A_k$ with entries 
		\begin{equation*}
			M(\sigma,\tau)=\begin{cases} 
				|\sigma_{k-1}\cap A_{k-1}| & \mbox{if $\sigma=\tau$}, \\
				(-1)^{\epsilon(\sigma,\tau)} &\mbox{if $\sigma\cap\tau\in A_{k-1}$}, \\
				0 & \mbox{otherwise},
			\end{cases}
		\end{equation*}
for $\sigma,\tau\in X_k\setminus A_k$. From the Rayleigh quotient theorem (cf. \cite[Theorem 4.2.2]{Hor13}) and the Ger\v{s}gorin circle theorem, we obtain
\begin{equation*}
			\begin{aligned}
				\mu_k(X)-\mu_k(X,A)&\leq
				\frac{\boldsymbol{y}^{\top}L_k(X)\boldsymbol{y}}{\boldsymbol{y}^{\top}\boldsymbol{y}}-\frac{\boldsymbol{x}^{\top}L_k(X,A)\boldsymbol{x}}{\boldsymbol{x}^{\top}\boldsymbol{x}}=\frac{\boldsymbol{x}^{\top}M\boldsymbol{x}}{\boldsymbol{x}^{\top}\boldsymbol{x}}\\
				&\leq\lambda_{\max}(M)\\
				&\leq \max_{\sigma\in X_k\setminus A_k}\left(|\sigma_{k-1}\cap A_{k-1}|+|\{\tau\in X_k\setminus A_k:\tau\cap\sigma\in A_{k-1}\}|\right),
			\end{aligned}
		\end{equation*}
and hence the result follows.
	\end{proof}

Finally, we give the proof of Theorem \ref{thm::main5}.

\renewcommand\proofname{\bf{Proof of Theorem \ref{thm::main5}}}	

\begin{proof}
	Recall that 
\[
B(X) = \{\sigma : \sigma \subseteq \tau \text{ for some } \tau \in X_{d-1} \text{ with } \deg_{X}(\tau) \leq 1\}
\]
is a $(d-1)$-dimensional subcomplex of $X$. Let $N$ be the diagonal matrix indexed by $X_d \setminus B(X)_d = X_d$ with entries
\begin{equation*}
    N(\sigma,\sigma) = |\sigma_{d-1} \cap B(X)_{d-1}|,
\end{equation*}
for each $\sigma \in X_d$. By Lemmas \ref{lem::2.1} and \ref{lem::2.2}, we obtain the decomposition 
\[
L_d(X) = L_d(X,B(X)) + N.
\]
Applying the Weyl inequality yields the following bounds for $\mu_d(X)$:
\begin{equation*}
   \mu_d(X,B(X)) + \min_{\sigma\in X_d} N(\sigma,\sigma) \leq \mu_d(X) \leq \mu_d(X,B(X)) + \max_{\sigma\in X_d} N(\sigma,\sigma).
\end{equation*}
Since $H_d(X,B(X);\mathbb{R})\neq 0$, by  Theorem \ref{thm::1.2},  we have $\mu_d(X,B(X)) = 0$, and hence the result follows. 
\end{proof}

Note that the condition $H_d(X,B(X); \mathbb{R}) \neq 0$ is easy to realize for simplicial complexes appearing in combinatorics. Below we construct an explicit example of a pure $d$-dimensional simplicial complex satisfying this condition. 
\begin{example}
	\rm 
Let $X$ be a simplicial complex such that we can find subcomplexes $X(i)$ for $1 \leq i \leq \ell$ satisfying the following properties:
\begin{enumerate}[(P1)]\setlength{\itemsep}{0pt}
    \item $X = \bigcup_{1 \leq i \leq \ell} X(i)$;
    \item $X(i)$ is homeomorphic to $D^d(i)$ as a topological space;
    \item $X(i) \cap X(j)$ is a subcomplex of dimension at most $d-2$ whenever $i \neq j$,
\end{enumerate}
where $D^d(i)$ is the $i$-th copy of the unit $d$-ball. Then, $X$ is a triangulation of the quotient topological space $\mathcal{X} = \mathcal{Y}/{\sim}$, where $\mathcal{Y}$ is the disjoint union of $D^d(i)$ for $1 \leq i \leq \ell$, and $\sim$ is the equivalence relation induced by the intersection relationships of the $X(i)$ in $X$. Let $\partial(\mathcal{X})$ denote the boundary of the topological space $\mathcal{X}$ (as a subspace of $\mathbb{R}^{d}$); then $B(X)$ corresponds to $\partial(\mathcal{X})$. Since the intersections of any two distinct $\mathcal{X}(i)$ are at most $(d-2)$-dimensional, these intersections lie entirely in $\partial(\mathcal{X})$. Therefore, $\mathcal{X}/\partial(\mathcal{X})\cong \mathcal{Y}/\partial(\mathcal{Y})$, which is homeomorphic to the wedge sum of $D^d(i)/\partial(D^d(i))\cong S^{d}$, where $S^d$ is the $d$-sphere, for $1 \leq i \leq \ell$. By \cite[Proposition 2.22, Corollary 2.25]{Hat02}, we have
\[
H_d(X,B(X);\mathbb{R}) \cong H_d(\mathcal{X},\partial(\mathcal{X});\mathbb{R}) \cong \widetilde{H}_d(\mathcal{X}/\partial(\mathcal{X});\mathbb{R}) \cong \mathbb{R}^\ell,
\]
and hence $H_d(X,B(X);\mathbb{R})\neq 0$.
\end{example}

Let $X$ be a pure $d$-dimensional simplicial complex. Then $X$ is called a \textit{$d$-path of length $m$} if there exists an ordering of its $d$-faces $\sigma_1,\ldots,\sigma_m$ such that $|\sigma_i \cap \sigma_j| = d$ if and only if $|j - i| = 1$. When $\sigma_m$ coincides with $\sigma_1$, we say that $X$ is a \textit{$d$-circuit of length $m$}. Furthermore, $X$ is called a \textit{$d$-star} if its $d$-faces can be arranged in a star-like formation, meaning that all $d$-faces share a common $(d-1)$-face. We say that $X$ is \textit{orientable} if we can assign an orientation to all $d$-faces of $X$ such that any two $d$-faces that intersect in a $(d-1)$-face induce opposite orientations on that face. Note that every $d$-path is orientable.

To conclude this section, we apply Theorem~\ref{thm::main5} to estimate the $d$-th spectral gaps of $d$-paths, $d$-cycles, and $d$-stars (see also \cite[Section~4]{HJ13}).

\begin{corollary} 
\begin{enumerate}[(i)]\setlength{\itemsep}{0pt}
  \item If $X$ is an orientable $d$-circuit, then $\mu_d(X) = d - 1$;
  \item If $X$ is a $d$-path, then $d - 1 \leq \mu_d(X) \leq d$;
  \item If $X$ is a $d$-star, then $\mu_d(X) = d$.
\end{enumerate}
\end{corollary}

\renewcommand\proofname{\bf{Proof}}
\begin{proof}
Let $X$ be an orientable $d$-circuit. By definition, we can assign an orientation to all $d$-faces of $X$ such that any two $d$-faces intersecting in a $(d-1)$-face induce opposite orientations on that face. Let $\alpha$ be the sum of all these oriented $d$-faces. Then, $\partial_d(\alpha) \in C_{d-1}(B(X); \mathbb{R})$, which implies $H_d(X, B(X); \mathbb{R}) \neq 0$. By Theorem \ref{thm::main5}, we obtain $\mu_d(X) = d - 1$, proving (i). For (ii), the proof is similar to that of (i), and so we omit it. For (iii), we take two distinct $d$-faces $\sigma,\tau\in X$. Then we see that either $\partial_d([\sigma]+[\tau])$ or $\partial_d([\sigma]-[\tau])$ belongs to $C_{d-1}(B(X);\mathbb{R})$, and so $H_d(X,B(X); \mathbb{R}) \neq 0$. By Theorem~\ref{thm::main5}, we immediately deduce that $\mu_d(S) = d$, and the result follows.
\end{proof}

	\section*{Declaration of Interest Statement}
	
	The authors declare that they have no known competing financial interests or personal relationships that could have appeared to influence the work reported in this paper.

	\section*{Acknowledgements}
	
	X. Huang was supported by the National Natural Science Foundation of China (No. 12471324) and the Natural Science Foundation of Shanghai (No. 24ZR1415500). L. Lu was supported by the National Natural Science Foundation of China (Grant No. 12371362) and  the Natural Science Foundation of Hunan Province  (Grant No. 2021JJ40707).
	
	\section*{Data Availability}
	No data was used for the research described in the article.

\end{document}